\documentclass[12pt]{article}
\usepackage{amsmath,amssymb,fullpage,amsthm,bbm,bm,graphicx,subcaption}
\usepackage[numbers]{natbib}

\newtheorem{theorem}{Theorem}[section]
\newtheorem{lemma}[theorem]{Lemma}
\newtheorem{remark}[theorem]{Remark}
\newtheorem{corollary}[theorem]{Corollary}
\newtheorem{proposition}[theorem]{Proposition}
\theoremstyle{definition}
\newtheorem{definition}{Definition}[section]

\numberwithin{equation}{section}
\newcommand{\bx}{\bm{x}}
\newcommand{\bX}{\bm{X}}

\DeclareMathOperator*{\argmax}{arg\,max}

\begin{document}
\begin{center}
\Large{\textbf{Bayesian mode and maximum estimation and accelerated rates of contraction}}
\end{center}

\begin{center}
\textbf{William Weimin Yoo and Subhashis Ghosal}\\
\textit{Leiden University and North Carolina State University}
\end{center}

\begin{abstract}
We study the problem of estimating the mode and maximum of an unknown regression function in the presence of noise. We adopt the Bayesian approach by using tensor-product B-splines and endowing the coefficients with Gaussian priors. In the usual fixed-in-advanced sampling plan, we establish posterior contraction rates for mode and maximum and show that they coincide with the minimax rates for this problem. To quantify estimation uncertainty, we construct credible sets for these two quantities that have high coverage probabilities with optimal sizes. If one is allowed to collect data sequentially, we further propose a Bayesian two-stage estimation procedure, where a second stage posterior is built based on samples collected within a credible set constructed from a first stage posterior. Under appropriate conditions on the radius of this credible set, we can accelerate optimal contraction rates from the fixed-in-advanced setting to the minimax sequential rates. A simulation experiment shows that our Bayesian two-stage procedure outperforms single-stage procedure and also slightly improves upon a non-Bayesian two-stage procedure.
\end{abstract}



\section{Introduction}
Consider noisy measurements $Y_1,\ldots,Y_n$ of an unknown smooth function $f$ at locations $\bX_1,\ldots,\bX_n\in [0,1]^d$ given by the nonparametric regression model
\begin{align}\label{eq:mainprob}
Y_i=f(\boldsymbol{X}_i)+\varepsilon_i,\quad i=1,\dotsc,n,
\end{align}
where the regression errors $\varepsilon_1,\ldots,\varepsilon_n$ are modeled as independent and identically distributed (i.i.d.) $\mathrm{N}(0,\sigma^2)$ with unknown standard deviation $0<\sigma<\infty$. The covariates can be deterministic or can be drawn as i.i.d samples from some fixed distribution independently of the regression errors.

In this paper, we consider the problem of estimating the mode $\bm{\mu}$ which marks the location of the maximum of $f$, and the value of this maximum $M=f(\bm{\mu})=\sup \{f(\bx): \bx\in [0,1]^d\}$, assuming that $\boldsymbol{\mu}$ is unique. The problem can be thought of as optimization in the presence of noise and has wide range of applications. For instance, searching for the optimal factor configurations in response surface methodology, locating peaks in bacteria (\citet{silverman}) and human (\citet{muller1985}) growth curves, or to classify and compare curves arising from longitudinal endocrinological data (\citet{jorgensen}).

The problem of estimating the mode and maximum of an isotropic regression function is well studied in the frequentist literature. M\"uller \cite{muller1985,muller1989} and Shoung and Zhang \cite{zhang2001} provided convergence rates for univariate regression, with the multivariate case obtained by \citet{muller2003}. Furthermore, Hasminski\u{i} \cite{hasminskii1979} and Tsybakov \cite{tsybakov1990} showed that for isotropic H\"{o}lder regression function of order $\alpha$ that is also $\alpha$-continuously differentiable, the minimax rates for estimating $\boldsymbol{\mu}$ is $n^{-(\alpha-1)/(2\alpha+d)}$ and for $M$ is $n^{-\alpha/(2\alpha+d)}$, under the usual sampling plan of choosing samples that are fixed in advance.

However if one is allowed to choose samples based on information gathered from past samples, the structure of the problem changes and we are in the sequential design setting. In this case, the minimax sequential rates of estimating $\boldsymbol{\mu}$ and $M$ are respectively $n^{-(\alpha-1)/(2\alpha)}$ and $n^{-1/2}$ (see \citet{max1988,tsybakov1990b,pelletier}). When compared with the fixed design case, it is clear that sequential rates are uniformly better and in fact $M$ has successfully achieved the parametric rate. Moreover, it also shows that judicious use of past information to guide future actions removes the effect of dimension $d$ on the rates. On the more practical side, Kiefer and Wolfowitz \cite{kiefer1952} and Blum \cite{blum1954} used Robbins-Monro type procedures that is consistent; while Fabian \cite{fabian1967}, Dippon \cite{dippon2003} and Mokkadem and Pelletier \cite{pelletier} each constructed sequential procedures that actually attain the minimax rates.

In actual practice, fully sequential design is costly to implement, because sample collection time is longer and the required logistics in collecting data in many stages is much more complicated than single-stage procedures. This then gave rise to the idea of a two-stage procedure, which offers a compromise between the added cost of doing a follow-up experiment and the added accuracy gained from it. At the first stage, limited samples are taken to give a pilot estimate of some quantity (e.g., mode), and the second stage samples are collected in the vicinity of this preliminary estimate. It was then shown in \citet{Lan,Tang} and \citet{2stage} that an extra second stage is enough to accelerate the convergence rates and in some cases propel them to attain the minimax sequential rates.

To the best of our knowledge however, there are hardly any such results and procedures in the Bayesian literature, whether it is in the fixed design, sequential or two-stage cases. Therefore, it is hoped that this paper will fill in this gap by giving a Bayesian solution to this problem. As we shall see, there are advantages in using the Bayesian approach, as it provides a natural framework to do two-stage estimation, and it can outperform frequentist procedures by exploiting the shrinkage property of Bayesian estimators.

In the first part of this article, we consider the fixed-in-advance sampling plan and establish single stage posterior contraction rates for $\boldsymbol{\mu}$ and $M$. Our prior consists of tensor product B-splines with Gaussian distributed coefficients, and we endow the error variance with some positive and continuous prior density. We chose this prior because it enables us to derive sharp results by directly analyzing the posterior distribution, and B-splines are efficient to compute. The main challenge here is the non-linear and non-smooth nature of the argmax and max functionals of $f$, and we avoid dealing with them directly by relating the estimation errors of $\boldsymbol{\mu}$ and $M$ with the sup-norm errors for $f$ and its first order partial derivatives. To quantify uncertainty in the estimation procedure, we construct credible sets for $\boldsymbol{\mu}$ and $M$, and show that they have high asymptotic coverage with optimal sizes.

Sequential sampling or more specifically a two-stage procedure can naturally be embedded inside a Bayesian framework, as information gained from an earlier stage can be used to adjust or update one's prior opinion. In the second part of this paper, we propose a Bayesian two-stage procedure for estimating the mode and maximum of $f$. We split the samples into two parts, and use the first part to compute the first stage posterior distribution of $\boldsymbol{\mu}$ and $M$. Using this posterior, we construct a credible set based on the techniques discussed in the first part of the paper. Second stage samples are then sampled uniformly over this set, and they are used to compute the second stage posterior of these two quantities.

We show that this second stage posterior is more concentrated around the truth than its single stage counterpart, and it can accelerate single stage minimax rates to the optimal sequential rates, under appropriate conditions on the radius of the credible set used. We test our procedure in a numerical experiment and the results seem to support our theoretical conclusions. Moreover when compared with a non-Bayesian method proposed in the literature, our Bayesian two-stage procedure seems to outperform slightly in terms of the root mean square error, and this is due to the shrinkage induced by our choice of prior distributions (see Figure \ref{fig:result} below).

Throughout this paper, we will work with a general class of anisotropic H\"{o}lder space, such that we allow $f$ to have different order of smoothness in each dimension. In some of our results below, it will be seen that additional smoothness in other dimensions can help alleviate the loss in accuracy due to less smoothness in some dimensions, and this borrowing of smoothness across dimensions, which is a unique feature of anisotropic spaces, can result in the improvement of the overall rate.

The paper is organized as follows. The next section introduces notations and assumptions. Section \ref{sec:tprior} describes the prior and the resulting posterior distributions of $\boldsymbol{\mu}$ and $M$. Section \ref{sec:tsupnormf} contains main results in the single stage setting on posterior contraction rates and coverage probability of credible sets for these two quantities. We introduce the Bayesian two-stage procedure of estimating $\boldsymbol{\mu}$ and $M$ in Section \ref{sec:bayes}. Section \ref{sec:sim} contains simulation studies for our proposed Bayesian two-stage method. This is then followed by a summary and discussion on future outlook in Section \ref{sec:conclude}. Proofs of our main results are given in Section \ref{sec:proof} and some useful auxiliary results are collected in Section \ref{sec:appendix} Appendix. We delegate some rather routine and technical proofs to a supplementary article \citet{yoo2018supp} to streamline reading.

\section{Notations and assumptions}\label{sec:notation}

Given two numerical sequences $a_n$ and $b_n$, $a_n=O(b_n)$ or $a_n\lesssim b_n$ means $a_n/b_n$ is bounded, while $a_n=o(b_n)$ or $a_n\ll b_n$ means $a_n/b_n\rightarrow0$. Also, $a_n\asymp b_n$ means $a_n=O(b_n)$ and $b_n=O(a_n)$. For stochastic sequence $Z_n$, $Z_n=O_{\mathrm{P}}(a_n)$ means $\mathrm{P}(|Z_n|\leq Ca_n)\rightarrow1$ for some constant $C>0$; while $Z_n=o_{\mathrm{P}}(a_n)$ means $Z_n/a_n\rightarrow0$ in $\mathrm{P}$-probability.

Let $\|\boldsymbol{x}\|_p=(\sum_{k=1}^d |x_k|^p)^{1/p}$, $\|\boldsymbol{x}\|_\infty=\max_{1\le k\le d} |x_k|$ and $\|\boldsymbol{x}\|=\|\boldsymbol{x}\|_2$. Inequality for a vector stands for co-ordinatewise inequality. For a symmetric matrix $\boldsymbol{A}$, let $\lambda_\mathrm{max}(\boldsymbol{A})$ and $\lambda_{\mathrm{min}}(\boldsymbol{A})$ stand for its largest and smallest eigenvalues, and $\|\boldsymbol{A}\|_{(2,2)}=|\lambda_{\mathrm{max}}(\boldsymbol{A})|$. Given another matrix $\boldsymbol{B}$ of the same size, $\boldsymbol{A}\leq\boldsymbol{B}$ means $\boldsymbol{B}-\boldsymbol{A}$ is nonnegative definite. The $L_p$-norm of a function $f$ is denoted by $\|f\|_p$.

We say $\boldsymbol{Z}\sim\mathrm{N}_J(\boldsymbol{\xi},\boldsymbol{\Omega})$ if $\boldsymbol{Z}$ has a $J$-dimensional normal distribution with mean $\boldsymbol{\xi}$ and covariance matrix $\boldsymbol{\Omega}$. By saying that $Z\sim\mathrm{GP}(\xi,\Omega)$, we mean that $\{Z(t),t\in U\}$ is a Gaussian process with $\mathrm{E}Z(t)=\xi(t)$ and $\mathrm{Cov}(Z(s),Z(t))=\Omega(s,t)$ for any $s,t\in U$.

Multi-indexes will be frequently used. Let $\mathbb{N}=\{1,2,\dotsc\}$ be the natural numbers and $\mathbb{N}_0=\mathbb{N}\cup\{0\}$. For $\boldsymbol{i}=(i_1,\dotsc,i_d)^T\in\mathbb{N}_0^d$ and  $\boldsymbol{x}\in\mathbb{R}^d$, define
$|\boldsymbol{i}|=\sum_{k=1}^di_k$, $\boldsymbol{i}!=\prod_{k=1}^di_k$ and  $\boldsymbol{x}^{\boldsymbol{i}}=\prod_{k=1}^dx_k^{i_k}$. For $\boldsymbol{r}=(r_1,\dotsc,r_d)^T\in\mathbb{N}_0^d$, let $D^{\boldsymbol{r}}=\partial^{|\boldsymbol{r}|}/\partial x_1^{r_1}\dotsm\partial x_d^{r_d}$ be the mixed partial derivative operator. If $\boldsymbol{r}=\boldsymbol{0}$, we interpret $D^{\boldsymbol{0}}f\equiv f$. If $\boldsymbol{r}=\boldsymbol{e}_k$, where $\boldsymbol{e}_k=(0,\dotsc,0,1,0,\dotsc,0)^T$ with $1$ in the $k$th position and zero elsewhere, we write $D^{\boldsymbol{e}_k}$ as $D_k$. We denote $\nabla f(\boldsymbol{x})=(D_1f(\boldsymbol{x}),\dotsc,D_df(\boldsymbol{x}))^T$ to be the gradient of $f$ at $\boldsymbol{x}$. If $f$ is twice differentiable,  $\boldsymbol{H}f(\boldsymbol{x}_0)$ stands for the Hessian matrix of $f$ at $\boldsymbol{x}_0$.

For $\boldsymbol{\alpha}=(\alpha_1,\dotsc,\alpha_d)^T\in\mathbb{N}^d$, let us denote $\alpha^{*}$ to be the harmonic mean, i.e., $(\alpha^{*})^{-1}=d^{-1}\sum_{k=1}^d\alpha_k^{-1}$. We define the anisotropic H\"{o}lder's norm $\|f\|_{\boldsymbol{\alpha},\infty}$ as
\begin{align}\label{eq:aninorm}
\max\left\{\|D^{\boldsymbol{r}}f\|_\infty+\sum_{k=1}^d\|D^{(\alpha_k-r_k)\boldsymbol{e}_k}D^{\boldsymbol{r}}f\|_\infty:\boldsymbol{r}\in\mathbb{N}_0^d,\quad\sum_{k=1}^d(r_k/\alpha_k)<1\right\}.
\end{align}
The constraint $\sum_{k=1}^d(r_k/\alpha_k)<1$ is a technical condition and is imposed so that contraction rates for $f$ and its derivatives will decrease to $0$ as $n\to\infty$.
\begin{definition}
The anisotropic H\"{o}lder space of order $\boldsymbol{\alpha}=(\alpha_1,\dotsc,\alpha_d)^T\in\mathbb{N}^d$, denoted as $\mathcal{H}^{\boldsymbol{\alpha}}([0,1]^d)$, consists of functions $f:[0,1]^d\rightarrow\mathbb{R}$ such that $\|f\|_{\boldsymbol{\alpha},\infty}<\infty$, and for some constant $C>0$ with any $\boldsymbol{x},\boldsymbol{x}_0\in(0,1)^d$,
\begin{align}\label{eq:isoholder}
|D^{\boldsymbol{r}}f(\boldsymbol{x})-D^{\boldsymbol{r}}T_{\boldsymbol{x}_0}f(\boldsymbol{x})|\leq C\sum_{k=1}^d|x_k-x_{0k}|^{\alpha_k-r_k},
\end{align}
where $\boldsymbol{r}\in\mathbb{N}_0^d$ and $\sum_{k=1}^d(r_k/\alpha_k)<1$. Here $T_{\boldsymbol{x}_0}f(\boldsymbol{x})=\sum_{\boldsymbol{i}\leq \boldsymbol{m}_{\boldsymbol{\alpha}}}D^{\boldsymbol{i}}f(\boldsymbol{x}_0)(\boldsymbol{x}-\boldsymbol{x}_0)^{\boldsymbol{i}}/\boldsymbol{i}!$ is the tensor Taylor polynomial of order $\boldsymbol{m}_{\boldsymbol{\alpha}}:=(\alpha_1-1, \alpha_2-1, \dotsc, \alpha_d-1)^T$ by expanding $f$ around $\boldsymbol{x}_0$.
\end{definition}

To study the frequentist properties of the posterior distribution, we assume the existence of a true regression function $f_0$ such that it satisfies the following three assumptions. In what follows, let  $\mathcal{B}(\boldsymbol{x},r)=\{\boldsymbol{y}: \|\boldsymbol{y}-\boldsymbol{x}\|\leq r \}$ be a $\ell_2$-ball of radius $r$ centered at $\bm{x}$.
\begin{enumerate}
	\item Under the true distribution $P_0$, $Y_i=f_0(\boldsymbol{X}_i)+\varepsilon_i$, such that $\varepsilon_i$ are i.i.d.~Gaussian with mean $0$ and variance $\sigma_0^2>0$ for $i=1,\dotsc,n$.
\item $f_0\in \mathcal{H}^{\boldsymbol{\alpha}}([0,1]^d)$ for $\alpha_k>2, k=1,\dotsc,d$, and attains its maximum $M_0$ at a unique point $\boldsymbol{\mu}_0$ in $(0,1)^d$ which is well-separated: for any constant $\tau_1>0$, there exists $\delta>0$ such that $f_0(\boldsymbol{\mu}_0)\geq f_0(\boldsymbol{x})+\delta$ for all $\boldsymbol{x}\notin\mathcal{B}(\boldsymbol{\mu}_0,\tau_1)$.
\item For any $0<\tau\leq\tau_1$, there exists $\lambda_0>0$ such that $\lambda_\text{max}\{\boldsymbol{H}f_0(\boldsymbol{x})\}<-\lambda_0$ for all $\boldsymbol{x}\in\mathcal{B}(\boldsymbol{\mu}_0,\tau)$.
\end{enumerate}

Assumption 1 states the true regression model for \eqref{eq:mainprob}. The well-separation property of Assumption 2 ensures that only points $\boldsymbol{x}$ that are near $\boldsymbol{\mu}_0$ will give values $f(\boldsymbol{x})$ that are close to the true maximum $M$. This property is needed to establish posterior consistency for $\boldsymbol{\mu}$ as we shall see in Theorem \ref{th:murate} below. Assumption 3 says that the Hessian of $f_0$ is locally negative definite around $\boldsymbol{\mu}_0$. Observe that Assumptions 2 and 3 imply $\nabla f_0(\boldsymbol{\mu}_0)=\boldsymbol{0}$ and the Hessian $\boldsymbol{H}f_0(\boldsymbol{\mu}_0)$ is symmetric and negative definite. Moreover, $\boldsymbol{H}f_0(\boldsymbol{x})$ is continuous in $\boldsymbol{x}$. If $\alpha_k=2$, then we need to make an extra assumption that the second partial derivatives of $f_0$ are continuous; if not, the Hessian may not be symmetric and its eigenvalues may not be real.

For $x_k\in[0,1]$, let $B_{j_k,q_k}(x_k)$ be the $k$th component B-spline of fixed order $q_k\geq\alpha_k$, with knots $0=t_{k,0}<t_{k,1}<\dotsb<t_{k,N_k}<t_{k,N_{k+1}}=1$, such that $J_k=q_k+N_k$. Assume that the set of knots in each direction is quasi-uniform, i.e., $\max_{1\leq l\leq N_k}(t_{k,l}-t_{k,l-1})\asymp\min_{1\leq l\leq N_k}(t_{k,l}-t_{k,l-1})$. Examples include uniform and nested uniform partitions (cf. Examples 6.6 and 6.7 of \citet{lschumaker}), and we can always choose a subset of knots from any given knot sequence to form a quasi-uniform sequence (cf. Lemma 6.17 of \citet{lschumaker}).

For fixed design points $\boldsymbol{X}_i=(X_{i1},\dotsc,X_{id})^T$ with $i=1,\dotsc,n$, assume that there is a cumulative distribution function $G$, with positive and continuous density $g$ on $[0,1]^d$ such that
\begin{align}\label{assump:cdf}
\sup_{\boldsymbol{x}\in[0,1]^d}|G_n(\boldsymbol{x})-G(\boldsymbol{x})|=o\left(\prod_{k=1}^dN_k^{-1}\right),
\end{align}
where $G_n(\boldsymbol{x})=n^{-1}\sum_{i=1}^n\mathbbm{1}_{\{\boldsymbol{X}_i\in [\boldsymbol{0},\boldsymbol{x}]\}}$ is the empirical distribution of $\{\boldsymbol{X}_i,i=1,\dotsc,n\}$, with $\mathbbm{1}_U$ the indicator function on $U$. The condition holds for the discrete uniform design with $G$ the uniform distribution when $N_k\lesssim n^{\alpha^{*}/\{\alpha_k(2\alpha^{*}+d)\}}$ for $k=1,\dotsc,d$. If $\boldsymbol{X}_i\overset{\mathrm{i.i.d.}}{\sim}G$ with a continuous density on $[0,1]^d$, then \eqref{assump:cdf} holds with probability tending to one if $N_k\lesssim n^{\alpha^{*}/\{\alpha_k(2\alpha^{*}+d)\}}$ for $k=1,\dotsc,d$, and $\alpha^{*}>d/2$ by Donsker's theorem. In this paper, we shall prove results on posterior contraction rates and credible sets based on fixed design points. These results will translate to the random case by conditioning on the predictor variables.

\section{B-splines tensor product, Gaussian prior and posterior}
\label{sec:tprior}
In the model $Y_i=f(\boldsymbol{x}_i)+\varepsilon_i$, $i=1,\ldots,n$, we put a finite random series prior on $f$ based on tensor-product B-splines, i.e., $f(\boldsymbol{x})=\sum_{j_1=1}^{J_1}\cdots\sum_{j_d=1}^{J_d}\theta_{j_1,\dotsc,j_d}\prod_{k=1}^dB_{j_k,q_k}(x_k):=\boldsymbol{b}_{\boldsymbol{J},\boldsymbol{q}}(\boldsymbol{x})^T\boldsymbol{\theta}$, where $\boldsymbol{b}_{\boldsymbol{J},\boldsymbol{q}}(\boldsymbol{x})=\{\prod_{k=1}^dB_{j_k,q_k}(x_k):1\leq j_k\leq J_k, k=1,\dotsc,d$\} is a collection of $J=\prod_{k=1}^dJ_k$ tensor-product B-splines, and $\boldsymbol{\theta}=\{\theta_{j_1,\dotsc,j_d}:1\leq j_k\leq J_k,k=1,\dotsc,d\}$ are the basis coefficients. Note that $\boldsymbol{b}_{\boldsymbol{J},\boldsymbol{q}}(\boldsymbol{x})$ and $\boldsymbol{\theta}$ are vectors indexed by $d$-dimensional indices and the entries are ordered lexicographically. Then by repeatedly applying equations (15) and (16) of Chapter X from \citet{deBoor} to each direction $k=1,\dotsc,d$, the $\boldsymbol{r}=(r_1,\dotsc,r_d)^T$ mixed partial derivative of $f$ is given by
\begin{equation}\label{eq:fprior}
D^{\boldsymbol{r}}f(\boldsymbol{x})=\sum_{j_1=1}^{J_1}\dotsi\sum_{j_d=1}^{J_d}\theta_{j_1,\dotsc,j_d}\prod_{k=1}^d\frac{\partial^{r_k}}{\partial x_k^{r_k}}B_{j_k,q_k}(x_k)=\boldsymbol{b}_{\boldsymbol{J},\boldsymbol{q}-\boldsymbol{r}}(\boldsymbol{x})^T\boldsymbol{W}_{\boldsymbol{r}}\boldsymbol{\theta},
\end{equation}
where $\boldsymbol{W}_{\boldsymbol{r}}$ is a $\prod_{k=1}^d(J_k-r_k)\times\prod_{k=1}^dJ_k$ matrix whose entries consist of coefficients associated with applying the finite difference operator  iteratively on $\boldsymbol{\theta}$ (for exact expressions, see (8.1)--(8.4) of \citet{yoo2016}). We represent the model in \eqref{eq:mainprob} by an $n$-variate normal distribution $\boldsymbol{Y}|(\boldsymbol{X},\boldsymbol{\theta},\sigma)\sim\mathrm{N}_n(\boldsymbol{B\theta},\sigma^2\boldsymbol{I}_n)$, where $\boldsymbol{B}=(\boldsymbol{b}_{\boldsymbol{J},\boldsymbol{q}}(\boldsymbol{X}_1)^T,\dotsc\boldsymbol{b}_{\boldsymbol{J},\boldsymbol{q}}(\boldsymbol{X}_n)^T)^T$ is the B-splines basis matrix. Note that we can index the rows and columns of $\boldsymbol{B}^T\boldsymbol{B}$ by multi-dimensional indices, such that for $\boldsymbol{u}=(u_1,\dotsc,u_d)^T$ and $\boldsymbol{v}=(v_1,\dotsc,v_d)^T$, we write $(\boldsymbol{B}^T\boldsymbol{B})_{\boldsymbol{u},\boldsymbol{v}}=\sum_{i=1}^n\prod_{k=1}^dB_{u_k,q_k}(X_{ik})B_{v_k,q_k}(X_{ik})$.

We consider deterministic $\boldsymbol{J}=(J_1,\dotsc,J_d)^T$ number of basis functions depending on $n, d$ and $\boldsymbol{\alpha}$. On the basis coefficients, we endow the prior $\boldsymbol{\theta}|\sigma^2\sim\mathrm{N}_J(\boldsymbol{\eta},\sigma^2\boldsymbol{\Omega})$. We choose the prior mean such that $\|\boldsymbol{\eta}\|_\infty<\infty$ and the $J\times J$ prior covariance matrix $c_1\boldsymbol{I}_J\leq\boldsymbol{\Omega}\leq c_2\boldsymbol{I}_J$ for some constants $0<c_1\leq c_2<\infty$. We will use the same multi-dimensional indexing convention of $\boldsymbol{B}^T\boldsymbol{B}$ on $\boldsymbol{\Omega}$, and further assume that $\boldsymbol{\Omega}^{-1}$ is $\boldsymbol{h}=(h_1,\dotsc,h_d)^T$-banded, in the sense that $(\boldsymbol{\Omega}^{-1})_{\boldsymbol{u},\boldsymbol{v}}=0$ if $|u_k-v_k|>h_k$ for some $k=1,\dotsc,d$.

By direct calculations, $D^{\boldsymbol{r}}f|(\boldsymbol{Y},\sigma)\sim\mathrm{GP}(\boldsymbol{A}_{\boldsymbol{r}}\boldsymbol{Y}+\boldsymbol{c}_{\boldsymbol{r}}\boldsymbol{\eta},\sigma^2\Sigma_{\boldsymbol{r}})$, where $\boldsymbol{A}_{\boldsymbol{r}}$, $\boldsymbol{c}_{\boldsymbol{r}}$ and the covariance kernel are defined for $\boldsymbol{x},\boldsymbol{y}\in(0,1)^d$ by
\begin{align}
\boldsymbol{A}_{\boldsymbol{r}}(\boldsymbol{x})&=\boldsymbol{b}_{\boldsymbol{J},\boldsymbol{q}-\boldsymbol{r}}(\boldsymbol{x})^T\boldsymbol{W}_{\boldsymbol{r}}\left(\boldsymbol{B}^T\boldsymbol{B}+\boldsymbol{\Omega}^{-1}\right)^{-1}\boldsymbol{B}^T,\label{eq:tAr}\\
\boldsymbol{c}_{\boldsymbol{r}}(\boldsymbol{x})&=\boldsymbol{b}_{\boldsymbol{J},\boldsymbol{q}-\boldsymbol{r}}(\boldsymbol{x})^T\boldsymbol{W}_{\boldsymbol{r}}\left(\boldsymbol{B}^T\boldsymbol{B}+\boldsymbol{\Omega}^{-1}\right)^{-1}\boldsymbol{\Omega}^{-1},\label{eq:tcr}\\
\Sigma_{\boldsymbol{r}}(\boldsymbol{x},\boldsymbol{y})&=\boldsymbol{b}_{\boldsymbol{J},\boldsymbol{q}-\boldsymbol{r}}(\boldsymbol{x})^T\boldsymbol{W}_{\boldsymbol{r}}
\left(\boldsymbol{B}^T\boldsymbol{B}+\boldsymbol{\Omega}^{-1}\right)^{-1}\boldsymbol{W}_{\boldsymbol{r}}^T\boldsymbol{b}_{\boldsymbol{J},\boldsymbol{q}-\boldsymbol{r}}
(\boldsymbol{y}).\label{eq:tsigmar}
\end{align}
For $\sigma$, we either take an empirical Bayes approach by maximizing the marginal likelihood obtained from  $\boldsymbol{Y}|\sigma\sim\mathrm{N}_n[\boldsymbol{B\eta},\sigma^2(\boldsymbol{B\Omega B}^T+\boldsymbol{I}_n)]$, or use a hierarchical Bayes approach. In the former approach, the empirical Bayes estimate given by $\widetilde{\sigma}_n^2=(\boldsymbol{Y}-\boldsymbol{B\eta})^T(\boldsymbol{B\Omega B}^T+\boldsymbol{I}_n)^{-1}(\boldsymbol{Y}-\boldsymbol{B\eta})/n$ is plugged into the expression of the conditional posterior process $D^{\boldsymbol{r}}f|(\boldsymbol{Y},\sigma)$; while in the latter approach, we further endow $\sigma$ with a continuous and positive prior density. For example, we can use a conjugate inverse-gamma (IG) prior $\sigma^2\sim\mathrm{IG}(\beta_1/2,\beta_2/2)$, where $\beta_1>4$ and $\beta_2>0$ are hyper-parameters, to get $\sigma^2|\boldsymbol{Y}\sim \mathrm{IG}\left[(\beta_1+n)/2,(\beta_2+n\widetilde{\sigma}_n^2)/2\right]$. The posterior for the function $f$ itself can be recovered as a special case $\boldsymbol{r}=\boldsymbol{0}$ by setting  $\boldsymbol{W}_{\boldsymbol{0}}=\boldsymbol{I}$.

\begin{remark}\rm
Finite random series based priors have been found to be very convenient to use in Bayesian nonparametrics because their theoretical and computational aspects can be dealt with very simply within the framework of Euclidean spaces. Even though they have simpler structure, they can achieve contraction rates on par with Gaussian process priors. Detailed discussions are given in \citet{shen2015} and the book \citet{ghosal2017}. More specifically, we used tensor-product B-splines with Gaussian coefficients because it enables us to lower bound the variation of the posterior around its center which is essential to get results on coverage of credible sets (to be discussed in Section \ref{sec:cred}).  The structure of the (Gaussian) prior is also helpful for bounding contraction rates and computing the posterior, and this allows us to obtain sharp rates and stronger statements regarding coverage probabilities (e.g.~without additional logarithmic factors in the radius), whether it is in the single or two-stage settings. Moreover, B-splines are compactly supported and we have fast recursive algorithm to compute them (see Section 5 of \citet{lschumaker}). We shall further discuss issues on adaptation in Section \ref{sec:conclude} where $\boldsymbol{\alpha}$ is not assumed to be known.
\end{remark}

We need to ensure that the priors discussed above for $f$ will yield a well-defined maximum point at every realization from its posterior. The following lemma assures this property.

\begin{lemma}
	\label{lem:uniquemu}
	If $f$ is given the tensor-product B-splines with normal coefficients prior, then $\boldsymbol{\mu}$ is unique for almost all sample paths of $f$ under the posterior distribution (empirical or hierarchical Bayes).
\end{lemma}

In this paper, we simply use $\Pi(\cdot|\boldsymbol{Y})$ to denote either the empirical or hierarchical posteriors, we do not distinguish between these two cases since both approaches yield the same rates.

\section{Main results for single stage setting}\label{sec:tsupnormf}
\subsection{Posterior contraction rates}

The posterior distributions of $\boldsymbol{\mu}$ and $M$ can be induced from the posterior of $f$ through the argmax and maximum operators. However since these operators are nonlinear and non-differentiable, we take an indirect approach by relating the estimation errors of $\boldsymbol{\mu}$ and $M$ to the sup-norm errors in estimating $f$ and its mixed partial derivatives. By this strategy, results for posterior contraction rates in the supremum norm can be used to induce the corresponding rates for $\boldsymbol{\mu}$ and $M$ as the following theorem shows.

\begin{theorem}
\label{th:murate}
Let $J_k\asymp(n/\log{n})^{\alpha^{*}/\{\alpha_k(2\alpha^{*}+d)\}}$, $k=1,\dotsc,d$ and assume that Assumptions 1, 2 and 3 hold. Consider the empirical Bayes approach by plugging-in $\widetilde{\sigma}_n$ for $\sigma$, or the hierarchical Bayes approach by equipping $\sigma$ with some continuous and positive prior density, then
\begin{eqnarray}
\label{eq:muinequality}
\|\boldsymbol{\mu}-\boldsymbol{\mu}_0\|&\leq& \frac{\sqrt{d}}{\lambda_0}\max_{1\leq k\leq d}\|D_kf-D_kf_0\|_\infty,\\
\label{eq:maxinequality}
|M-M_0| &\leq & \|f-f_0\|_\infty.
\end{eqnarray}
Therefore for any $m_n\rightarrow\infty$ and uniformly in $\|f_0\|_{\boldsymbol{\alpha},\infty}\leq R,R>0$,
\begin{align}
&\mathrm{E}_0\Pi\left(\|\boldsymbol{\mu}-\boldsymbol{\mu}_0\|>m_n(\log{n}/n)^{\alpha^{*}\{1-(\min_{1\leq k\leq d}\alpha_k)^{-1}\}/(2\alpha^{*}+d)}\middle|\boldsymbol{Y}\right)\rightarrow0,
\label{eq:rate mu}\\
&\mathrm{E}_0\Pi\left(|M-M_0|>m_n(\log{n}/n)^{\alpha^{*}/(2\alpha^{*}+d)}\middle|\boldsymbol{Y}\right)\rightarrow0.
\label{eq:rate M}
\end{align}
\end{theorem}

Clearly, a consequence of the result above is that the posterior mean $\mathrm{E}(\boldsymbol{\mu}|\boldsymbol{Y})$ converges to $\bm{\mu}_0$ in $\ell_2$-norm at the same rate given in \eqref{eq:rate mu}, and the same can be said for $\mathrm{E}(M|\boldsymbol{Y})$. Given the absence of minimax results on anisotropic mode estimation (to the best of our knowledge), it is instructive to ask whether the inequalities used above are sharp and the contraction rates obtained are optimal? The following lower bound corollary shows that these results are sharp up to logarithmic factors.

\begin{corollary}[Lower Bounds]\label{cor:lower}
In addition to Assumptions 1,2 and 3, let us now make an extra assumption that for any $0<\tau\leq\tau_1$, we have
\begin{align}\label{assump:lower}
\inf_{\boldsymbol{x}:\|\boldsymbol{x}-\boldsymbol{\mu}_0\|\leq\tau}\lambda_{\mathrm{min}}\{\boldsymbol{H}f_0(\boldsymbol{x})\}>-\lambda_1,
\end{align}
for some constants $\tau_1,\lambda_1>0$. Then for some small enough constant $h>0$, we have
\begin{align*}
&\mathrm{E}_0\Pi\left(\|\boldsymbol{\mu}-\boldsymbol{\mu}_0\|\geq hn^{-\alpha^{*}\{1-(\min_{1\leq k\leq d}\alpha_k)^{-1}\}/(2\alpha^{*}+d)}\middle|\boldsymbol{Y}\right)\rightarrow1,\\
&\mathrm{E}_0\Pi\left(|M-M_0|\geq hn^{-\alpha^{*}/(2\alpha^{*}+d)}\middle|\boldsymbol{Y}\right)\rightarrow1.
\end{align*}
\end{corollary}

For the isotropic smooth case $\alpha_1=\cdots=\alpha_d=\alpha$, the norm in \eqref{eq:aninorm} can be generalized (see Section 2.7.1 of \citet{empirical}) and the B-splines approximation error rate is obtained for all smoothness levels (Theorem 22 of Chapter XII in \citet{deBoor}). This allows generalization of these and subsequent results in this paper for arbitrary smoothness levels. The contraction rates thus obtained, when reduced to the isotropic case, are the minimax rates for this problem up to some logarithmic factor (see \citet{hasminskii1979} and \citet{tsybakov1990} as discussed in the introduction). In Section \ref{sec:bayes} below, we will describe another Bayesian procedure that is able to remove this logarithmic factor and accelerate these rates to the optimal sequential rates.

\begin{remark}
\rm
\label{new remark 1}

In the rates above, clearly the direction which has the least smoothness is the most influential factor in determining the contraction rate for $\boldsymbol{\mu}$ because of the presence of the second factor in the numerator of the exponent. This is unlike the contraction rate for $f$, which is known to be $(\log{n}/n)^{\alpha^{*}/(2\alpha^{*}+d)}$ (Theorem 4.4 of \citet{yoo2016}), as it depends only on the harmonic mean $\alpha^*$ of smoothness. The reason is evident from \eqref{eq:muinequality} in that the largest of the deviations of the function's derivative across all directions bounds the accuracy of estimating $\bm{\mu}$. Nevertheless, it is easily checked that the rate obtained above is better than that obtained by applying the above result on a function of isotropic smoothness $\min_{1\leq k\leq d}\alpha_k$. In other words, additional smoothness in other directions can help to alleviate the comparative loss of accuracy for dimensions which are less smooth, and this borrowing of smoothness across directions, which is a unique feature to anisotropic spaces, results in the improvement of the overall rate.
\end{remark}

\subsection{Credible regions for mode and maximum}
\label{sec:cred}

Let us now quantify uncertainty by constructing credible regions for $\boldsymbol{\mu}$ and $M$. In what follows, we require that these sets have credibility at least some given level $1-\gamma$, and they have optimal sizes with asymptotic coverage probability also at least $1-\gamma$.

We first construct credible sets in the form of supremum-norm balls in the space of regression functions, and then we map these regions back using the argmax and max functionals into Euclidean spaces, so that they are credible sets for $\boldsymbol{\mu}$ and $M$. Now, the natural Bayesian approach to this problem is to directly construct these sets from the posterior distributions of $\boldsymbol{\mu}$ and $M$. The main reason for favoring the proposed method is that it allows tighter control over the size of the induced credible regions in view of \eqref{eq:muinequality} and \eqref{eq:maxinequality}. Such a control is essential for frequentist coverage, and enables us to use them later in the Bayesian two-stage procedure in Section \ref{sec:bayes}.

We make a remark before we present the main result of this section. It is well known that in nonparametric models, a credible region may have frequentist coverage asymptotically less than the corresponding credibility level. The asymptotic coverage may even be arbitrarily close to zero, because the order of bias of the center of a credible set may be comparable with its variation around the truth under optimal smoothing; see \citet{cox1993,freedman1999} and \citet{inverseprob}. This is in sharp contrast with finite dimensional models where Bayesian and frequentist quantification of uncertainty agree because of the Bernstein-von Mises theorem. Knapik {\em et al.} \cite{inverseprob} showed that this low coverage problem can be addressed by undersmoothing. Castillo and Nickl \cite{castillo2013nonparametric,castillo2014bernstein} circumvented this problem by using weaker norms to construct credible sets. Szabo {\em et al.} \cite{botond2015} and Yoo and Ghosal \cite{yoo2016} addressed this same problem by appropriately inflating the size of credible regions to ensure coverage. In our own construction, we shall use the latter approach by introducing a constant $\rho>0$ in the radius and choose it large enough so that we will have asymptotic coverage.

Below by posterior distribution we refer to either the empirical Bayes posterior distribution by substituting $\widetilde\sigma_n$ for $\sigma$, or the hierarchical Bayes posterior distribution obtained by putting a further prior on $\sigma$. Denote the posterior mean of $f$ by $\widetilde f$, and let $\widetilde{\bm{\mu}}$ be the mode of $\widetilde f$ and $\widetilde M$ its maximum value. For $e_{\boldsymbol{k}}=(0,\dotsc,0,1,0,\dotsc,0)$ with $1$ at entry $k$ and the rest zero, we abbreviate $\boldsymbol{A}_{\boldsymbol{e}_k}$ by $\boldsymbol{A}_k$, $\boldsymbol{c}_{\boldsymbol{e}_k}$ by $\boldsymbol{c}_k$ and $\Sigma_{\boldsymbol{e}_k}$ as $\Sigma_k$ respectively for any $k=1,\dotsc,d$, where these quantities were defined in \eqref{eq:tAr}--\eqref{eq:tsigmar}.

\begin{remark}\label{rem:uniquemutilde}\rm
Note that $\widetilde{\boldsymbol{\mu}}$ is well-defined under $P_0$. Indeed since the posterior mean is an affine transformation of $\boldsymbol{Y}$, it follows from Assumption 1 that $\widetilde{f}$ is a Gaussian process under $P_0$. Therefore using the same argument as in the proof of Lemma \ref{lem:uniquemu}, we see that $\widetilde{f}$ has unique maximum $\widetilde{\boldsymbol{\mu}}$ for every realization.
\end{remark}

For some $0<\gamma<1/2$, consider $\{f:\|D_kf-\bm{A}_k \bm{Y}-\bm{c}_k \bm{\eta}\|_\infty\leq\rho R_{n,k,\gamma}\}$ as a credible band for $f$, where $\rho>0$ is a sufficiently large constant and $R_{n,k,\gamma}$ is the $(1-\gamma)$-quantile of the posterior distribution of $\|D_k f- \bm{A}_k \bm{Y}-\bm{c}_k \bm{\eta}\|_\infty$. Similarly, let $R_{n,\bm{0},\gamma}$ be the $(1-\gamma)$-quantile of the posterior distribution of $\|f- \bm{A}_{\bm{0}} \bm{Y}-\bm{c}_{\bm{0}} \bm{\eta}\|_\infty$. We proceed by using these sets to induce credible regions for $\bm{\mu}$ and $M$ through the argmax and maximum functionals, and they are given by
\begin{eqnarray}
{\mathcal{C}}_{\boldsymbol{\mu}} &=& \bigcap_{k=1}^d \left\{\boldsymbol{\mu}: \|D_k f-\boldsymbol{A}_k\boldsymbol{Y}-\boldsymbol{c}_k\boldsymbol{\eta}\|_\infty\leq \rho R_{n,k,\gamma} \right\},
\label{eq:mu credible}\\
{\mathcal{C}}_{M}&=&\left\{M : \|f-\boldsymbol{A}_{\bm{0}}\boldsymbol{Y}-\boldsymbol{c}_{\bm{0}}\boldsymbol{\eta}\|_\infty\leq \rho R_{n,\bm{0},\gamma}
\right\}.
\label{eq:M credible}
\end{eqnarray}
The following result establishes properties of these regions.

\begin{theorem}
\label{th:crmurho}
If $J_k\asymp(n/\log{n})^{\alpha^{*}/\{\alpha_k(2\alpha^{*}+d)\}}$, $k=1,\dotsc,d$, then we have uniformly in $\|f_0\|_{\boldsymbol{\alpha},\infty}\leq R$ for any $R>0$:
\begin{enumerate}
	\item [{\rm (i)}] the credibility of ${\mathcal{C}}_{\boldsymbol{\mu}}$ tends to $1$ in $P_0$-probability and its coverage approaches $1$ asymptotically,
\item [{\rm (ii)}] ${\mathcal{C}}_{\boldsymbol{\mu}}\subset\overline{\mathcal{C}}_{\boldsymbol{\mu}}:=\{\bm{\mu}: \|\bm{\mu}-\widetilde{\bm{\mu}}\|_\infty\leq \rho\sqrt{d}\lambda_0^{-1}\max_{1\leq k\leq d}R_{n,k,\gamma}\}$ with $P_0$-probability going to $1$,
\item [{\rm (iii)}]$ \underline{\mathcal{C}}_{\bm{\mu}}:=\{\bm{\mu}: \|\bm{\mu}-\widetilde{\bm{\mu}}\|_\infty\leq (Rd)^{-1}\rho\max_{1\leq k\leq d}R_{n,k,\gamma}\}\subset\mathcal{C}_{\boldsymbol{\mu}}$ with $P_0$-probability tending to $1$,
\item [{\rm (iv)}]  the credibility of ${\mathcal{C}}_M$ tends to $1$ in $P_0$-probability and its coverage approaches $1$ asymptotically,
\item [{\rm (v)}] $\mathcal{C}_M\subset \{M: |M-\widetilde{M}|\le \rho R_{n,\bm{0},\gamma}\}$.
\end{enumerate}
\end{theorem}
Assertions (ii) and (iii) say that the induced credible set $\mathcal{C}_{\boldsymbol{\mu}}$ can be sandwiched between two hypercubes, and its size is not too small when compared with the upper bound in (ii). Thus, its radius is of the order $\max_{1\leq k\leq d}R_{n,k,\gamma}\asymp\max_{1\leq k\leq d}(\log{n}/n)^{\alpha^{*}(1-\alpha_k^{-1})/(2\alpha^{*}+d)}$ by the second statement of Theorem \ref{th:fcred} in Section \ref{sec:appendix} Appendix (with $\boldsymbol{r}=\boldsymbol{e}_k$); while (v) and the same aforementioned statement (with $\boldsymbol{r}=\boldsymbol{0}$) imply that the radius of $\mathcal{C}_M$ is of the order $(\log{n}/n)^{\alpha^{*}/(2\alpha^{*}+d)}$. Note that these radius lengths coincide exactly with the contraction rates of Theorem \ref{th:murate}.

The result above concludes that the induced credible regions for $\bm{\mu}$ and $M$, i.e., $\mathcal{C}_{\boldsymbol{\mu}}$ and $\mathcal{C}_M$ respectively, have adequate frequentist coverage that are of (nearly) optimal sizes. Assertion (ii) also implies that the hypercube $\overline{\mathcal{C}}_{\bm{\mu}}$ centered at $\widetilde{\bm{\mu}}$ has at least $(1-\gamma)$-credibility and is a confidence set of nearly optimal size. Thus a credible set with guaranteed frequentist coverage can be chosen to be a simple set like a hypercube centered at the posterior mean. In practice, it is easier to construct such a hypercube than the set  $\mathcal{C}_{\boldsymbol{\mu}}$, because the latter set requires performing function maximization multiple times to obtain points in $\mathcal{C}_{\boldsymbol{\mu}}$.

The construction of $\overline{\mathcal{C}}_{\bm{\mu}}$ from the data is simple: one finds $b$ such that the credibility of $\{\bm{\mu}: \|\bm{\mu}-\widetilde{\bm{\mu}}\|_\infty\le b\}$ is $1-\gamma$, and then inflates this around $\widetilde{\bm{\mu}}$ by a large constant factor $\rho>0$. For this set to serve as the domain for second stage sampling in a two-stage procedure, some modifications are needed, in that we adjust the length of $\overline{\mathcal{C}}_{\boldsymbol{\mu}}$ in each direction so that it adapts to different smoothness. In other words, we embed $\overline{\mathcal{C}}_{\boldsymbol{\mu}}$ inside a hyper-rectangle and do uniform sampling inside this larger set. Clearly, keeping the constant inflation factor as small as possible makes the credible sets smaller, but it will be seen that for optimal contraction rate in the second stage, an inflation factor which goes to infinity at a specific rate will be needed.

\section{Two-stage Bayesian estimation and accelerated rates of contraction}\label{sec:bayes}

In this section we show that by obtaining samples in two stages in an appropriate manner, we can accelerate the posterior contraction rates of $\boldsymbol{\mu}$ and $M$ to the optimal sequential rates. Given a sampling budget of $n$, we first obtain $n_1<n$ samples to compute the first stage posterior distribution. The remaining $n_2=n-n_1$ samples are then obtained by sampling points uniformly from some regularly shaped credible region constructed from this posterior. Since this is a small region, we can approximate the regression function $f$ by a multivariate polynomial. By further endowing the coefficients with normal priors, we then use these samples to build the second stage posterior distribution for $f$ and hence for $\boldsymbol{\mu}$ and $M$ through the argmax and maximum functionals. We will then show through Theorem \ref{th:2M} below and simulations (Section \ref{sec:sim}) that these second stage posteriors are more concentrated near the truth.

Let us first describe our proposed Bayesian two-stage procedure in greater detail. Let $p\in(0,1)$. In the first stage, we choose $n_1\in\mathbb{N}$ design points $\{\widetilde{\boldsymbol{x}}_i,i=1,\dotsc,n_1\}$ such that $n_1/n\rightarrow p$ as $n\rightarrow\infty$ to obtain data $\mathcal{D}_1=\{(\widetilde{\boldsymbol{x}}_i,\widetilde{Y}_i),i=1,\dotsc,n_1\}$ for the model in \eqref{eq:mainprob}. Typically, one chooses $p=1/2$ to achieve equal sample splitting but other proportions are possible depending on the sampling configurations (e.g., grid or random sampling) and other practical considerations such as field conditions and financial constraints. By using the B-spline tensor product prior discussed in Section \ref{sec:tprior}, we obtain a first stage posterior for $f$. This then allows us to construct the set
\begin{align*}
\{\boldsymbol{\mu}:|\mu_k-\widetilde{\mu}_k|\leq\delta_{n,k},k=1,\dotsc,d\},
\end{align*}
where $\widetilde{\boldsymbol{\mu}}$ is the mode of the first stage posterior mean of $f$. We choose $\delta_{n,k},k=1,\dotsc,d$ such that
\begin{align}\label{eq:delta}
\min_{1\leq k\leq d}\delta_{n,k}=\rho_n\max_{1\leq k\leq d}(\log{n}/n)^{\alpha^{*}(1-\alpha_k^{-1})/(2\alpha^{*}+d)}
\end{align}
for a chosen sequence $\rho_n\rightarrow\infty$, so that this set is a valid credible set, as it contains $\overline{\mathcal{C}}_{\boldsymbol{\mu}}$ for large $n$. Now sample $n_2=n-n_1$ locations $\{\boldsymbol{x}_1,\dotsc,\boldsymbol{x}_{n_2}\}$ uniformly from this credible set and observe the second stage samples $\mathcal{D}_2=\{(\boldsymbol{x}_i,Y_i),i=1,\dotsc,n_2\}$.

Next, we center the second stage design points at the origin by $\boldsymbol{z}_i=\boldsymbol{x}_i-\widetilde{\boldsymbol{\mu}}$ for $i=1,\dotsc,n_2$. In other words, $\boldsymbol{z}_i,i=1,\dotsc,n_2$, are i.i.d.~uniform samples from the hyper-rectangle $\mathcal{Q}:=\{\bm{x}: |x_k|\leq\delta_{n,k}, k=1,\dotsc,d\}$ of sides $\delta_{n,k},k=1,\dotsc,d$. Observe that $\boldsymbol{z}_i,i=1,\dotsc,n_2$, are independent from the errors $\boldsymbol{\varepsilon}$ in this sampling scheme. We chose this sampling domain because we need its length at each direction to adapt to different smoothness, and it is operationally more convenient to construct credible sets in the form of hyper-rectangles and do uniform sampling on it (see Remark \ref{rem:credible} below for a more thorough discussion).

At the second stage, we put a prior on the regression function by representing $f(\boldsymbol{z})$ at $\boldsymbol{z}=(z_1,\dotsc,z_d)^T\in\mathcal{Q}$ as a multivariate polynomial function of fixed order $\boldsymbol{m}_{\boldsymbol{\alpha}}=(\alpha_1-1,\dotsc,\alpha_d-1)^T$, i.e., for $\boldsymbol{z}^{\boldsymbol{i}}=\prod_{k=1}^dz_k^{i_k}$,
\begin{align}\label{eq:thetaf}
f_{\boldsymbol{\theta}}(\boldsymbol{z})=\sum_{\boldsymbol{i}\leq\boldsymbol{m}_{\boldsymbol{\alpha}}}\theta_{\boldsymbol{i}}\boldsymbol{z}^{\boldsymbol{i}}=\boldsymbol{p}(\boldsymbol{z})^T\boldsymbol{\theta},
\end{align}
where $\boldsymbol{p}(\boldsymbol{z})=(\boldsymbol{z}^{\boldsymbol{i}}:\boldsymbol{i}\leq\boldsymbol{m}_{\boldsymbol{\alpha}})^T$ and $\boldsymbol{\theta}=(\theta_{\boldsymbol{i}}:\boldsymbol{i}\leq\boldsymbol{m}_{\boldsymbol{\alpha}})^T$ are the corresponding basis coefficients. The elements of $\{\boldsymbol{i}:\boldsymbol{i}\leq\boldsymbol{m}_{\boldsymbol{\alpha}}\}$ can be enumerated as $\{\boldsymbol{i}_0,\boldsymbol{i}_1,\dotsc,\boldsymbol{i}_W\}$ where $W+1=\prod_{k=1}^d\alpha_k$ with $\boldsymbol{i}_0=\boldsymbol{0}$. Define $\boldsymbol{Z}=(\boldsymbol{p}(\boldsymbol{z}_1),\dotsc,\boldsymbol{p}(\boldsymbol{z}_{n_2}))^T$, and note that for $d=1$, $\boldsymbol{Z}$ is a Vandermonde matrix.

We endow $\boldsymbol{\theta}$ with the prior $\boldsymbol{\theta}|\sigma^2\sim\mathrm{N}_{W+1}(\boldsymbol{\xi},\sigma^2\boldsymbol{V})$, where the entries of $\boldsymbol{\xi}$ do not depend on $n$ and $\boldsymbol{V}=\mathrm{diag}\left\{\prod_{k=1}^d\delta_{n,k}^{-2(\boldsymbol{i}_j)_k}:j=0,1\ldots,W\right\}$. Then it follows that the posterior $\Pi(\boldsymbol{\theta}|\boldsymbol{Y},\sigma^2)$ is
\begin{align}\label{eq:ptheta}
\mathrm{N}_{W+1}\left[(\boldsymbol{Z}^T\boldsymbol{Z}+\boldsymbol{V}^{-1})^{-1}(\boldsymbol{Z}^T\boldsymbol{Y}
+\boldsymbol{V}^{-1}\boldsymbol{\xi}),
\sigma^2(\boldsymbol{Z}^T\boldsymbol{Z}+\boldsymbol{V}^{-1})^{-1}\right].
\end{align}
The empirical posterior follows by replacing $\sigma^2$ with $\widetilde{\sigma}_{*}^2=(n_1\widetilde{\sigma}_1^2+n_2\widetilde{\sigma}_2^2)/n$, where $\widetilde{\sigma}_1^2=(\widetilde{\boldsymbol{Y}}-\boldsymbol{B\eta})^T(\boldsymbol{B\Omega B}^T+\boldsymbol{I}_{n_1})^{-1}(\widetilde{\boldsymbol{Y}}-\boldsymbol{B\eta})/n_1$ is the empirical estimate of $\sigma^2$ based on the first stage samples, and $\widetilde{\sigma}_2^2=n_2^{-1}(\boldsymbol{Y}-\boldsymbol{Z\xi})^T(\boldsymbol{ZVZ}^T+\boldsymbol{I}_{n_2})^{-1}
(\boldsymbol{Y}-\boldsymbol{Z\xi})$ is the same estimate based on the second stage samples.

For the hierarchical Bayes approach, we use the first stage posterior of $\sigma^2$ as prior for the second stage. That is, we equip $\sigma^2$ with $\mathrm{IG}(\beta_1/2,\beta_2/2)$ prior at the first stage, where $\beta_1>4$ and $\beta_2>0$, we then use the resulting posterior $\mathrm{IG}[(\beta_1+n_1)/2,(\beta_2+n_1\widetilde{\sigma}_1^2)/2]$ as prior for the second stage, which will further yield $\mathrm{IG}[(\beta_1+n)/2, (\beta_2+n\widetilde{\sigma}_{*}^2)/2]$ as the second stage posterior for $\sigma^2$. The proposition below shows that the second stage empirical Bayes estimator and the hierarchical Bayes posterior of $\sigma^2$ are consistent, and it is a key step in establishing the main result given in Theorem \ref{th:2M} below.

\begin{proposition}[Second stage error variance]\label{prop:var2}
Uniformly over $\|f_0\|_{\boldsymbol{\alpha},\infty}\leq R$,
\begin{itemize}
\item [(a)] The second stage empirical Bayes estimator $\widetilde{\sigma}_{*}^2$ converges to $\sigma_0^2$ in $P_0$-probability at the rate $\max\{n^{-1/2}, n^{-2\alpha^{*}/(2\alpha^{*}+d)}, \sum_{k=1}^d\delta_{n,k}^{2\alpha_k}\}$.
\item [(b)] If inverse gamma posterior from the first stage is used as the prior in the second stage, the second stage posterior of $\sigma^2$ contracts to $\sigma_0^2$ at the same rate.
\end{itemize}
\end{proposition}

Let $\boldsymbol{r}=(r_1,\dotsc,r_d)^T$ be such that $\boldsymbol{r}\leq\boldsymbol{m}_{\boldsymbol{\alpha}}$. Then the $\boldsymbol{r}$ mixed partial derivative of $f_{\boldsymbol{\theta}}$ is
\begin{equation}\label{eq:polyprime}
D^{\boldsymbol{r}}f_{\boldsymbol{\theta}}(\boldsymbol{z})=\sum_{\boldsymbol{i}\leq\boldsymbol{m}_{\boldsymbol{\alpha}}}\theta_{\boldsymbol{i}}\prod_{k=1}^d\frac{\partial^{r_k}}{\partial z_k^{r_k}}z_k^{i_k}=\sum_{\boldsymbol{r}\leq\boldsymbol{i}\leq\boldsymbol{m}_{\boldsymbol{\alpha}}}\theta_{\boldsymbol{i}}\frac{\boldsymbol{i}!}{(\boldsymbol{i}-\boldsymbol{r})!}\boldsymbol{z}^{\boldsymbol{i}-\boldsymbol{r}},
\end{equation}
which is a multivariate polynomial of degree $\boldsymbol{m}_{\boldsymbol{\alpha}}-\boldsymbol{r}$. The posterior distributions of $D^{\boldsymbol{r}}f_{\boldsymbol{\theta}}$ can then be induced from \eqref{eq:ptheta}.

Let us define the location of the maximum of $f_{\boldsymbol{\theta}}$ inside the centered region as $\boldsymbol{\mu}_{\boldsymbol{z}}=\argmax_{\boldsymbol{z}\in \mathcal{Q}}f_{\boldsymbol{\theta}}(\boldsymbol{z})$. We then relate this location back to the original domain by $\boldsymbol{\mu}=\widetilde{\boldsymbol{\mu}}+\boldsymbol{\mu}_{\boldsymbol{z}}$ with corresponding maximum value $M=f_{\boldsymbol{\theta}}(\boldsymbol{\mu}_{\boldsymbol{z}})$. Following the same reasoning as in Lemma \ref{lem:uniquemu}, $\boldsymbol{\mu}$ is unique for almost all sample paths of $f_{\boldsymbol{\theta}}$ under the empirical or hierarchical posterior. The following theorem establishes the second stage posterior contraction rates of $\boldsymbol{\mu}$ and $M$ for any smoothness level $\alpha_k>2,k=1,\dotsc,d$, uniformly over $\|f_0\|_{\boldsymbol{\alpha},\infty}\leq R$.

\begin{theorem}\label{th:2M}
For any chosen sequence $\rho_n\to \infty$, let $\delta_{n,k},k=1,\dotsc,d$, be such that $\min_{1\leq k\leq d}\delta_{n,k}=\rho_n\max_{1\leq k\leq d}(\log{n}/n)^{\alpha^{*}(1-\alpha_k^{-1})/(2\alpha^{*}+d)}$. Then under Assumptions 1, 2 and 3, we have uniformly over $\|f_0\|_{\boldsymbol{\alpha},\infty}\leq R$ and for any $m_n\rightarrow\infty$,
\begin{align*}
\mathrm{E}_0\Pi\left[\|\boldsymbol{\mu}-\boldsymbol{\mu}_0\|>m_n\max_{1\leq k\leq d}\delta_{n,k}^{-1}\left(n^{-1/2}+\sum_{l=1}^d\delta_{n,l}^{\alpha_l}\right)\middle|\boldsymbol{Y}\right]\rightarrow0,\\
\mathrm{E}_0\Pi\left[|M-M_0|>m_n\left(n^{-1/2}+\sum_{k=1}^d\delta_{n,k}^{\alpha_k}\right)\middle|\boldsymbol{Y}\right]\rightarrow0.
\end{align*}
In particular, if $\alpha_k>1+\sqrt{1+d/2}$ for all $k=1,\dotsc,d$, then for the choice $\delta_{n,k}=n^{-1/(2\alpha_k)}$,  $k=1,\dotsc,d$, the posterior distributions for $\bm{\mu}$ and $M$ contract at the rates $n^{-(\underline{\alpha}-1)/(2\underline{\alpha})}$ and $n^{-1/2}$ respectively, where $\underline{\alpha}=\min_{1\leq k\leq d}\alpha_k$.
\end{theorem}

Let us take $\delta_{n,k}=n^{-1/(2\alpha_k)}$ the optimal choice suggested above. By comparing this theorem and the single stage contraction rates of Theorem \ref{th:murate}, we can draw the following three main conclusions:
\begin{enumerate}
\item If we perform a Bayesian two-stage procedure, we accelerate contraction rates for estimating $\boldsymbol{\mu}$ and $M$, with $M$ achieving the parametric rate $n^{-1/2}$.
\item At the same time, we remove extra logarithmic factors that are present in the single stage rates.
\item The second stage rates for $\boldsymbol{\mu}$ and $M$ do not depend on $d$ the dimension of the regression function's domain, and the effect of dimension is mitigated to a lower bound $1+\sqrt{1+d/2}$ required on the smoothness at each direction.
\end{enumerate}
The first conclusion says that the second stage posteriors for $\boldsymbol{\mu}$ and $M$ are more concentrated near the truth when compared with their single stage counterparts, and this is evident since the second stage rate of $\boldsymbol{\mu}$ is $n^{-(\underline{\alpha}-1)/(2\underline{\alpha})}\ll(\log{n}/n)^{\alpha^{*}(1-\underline{\alpha}^{-1})/(2\alpha^{*}+d)}$; while for $M$ is $n^{-1/2}\ll(\log{n}/n)^{\alpha^{*}/(2\alpha^{*}+d)}$. As noted above, $M$ has achieved its oracle or the parametric rate. The second stage rate for $\boldsymbol{\mu}$ is sharp, to see this note that if $k$ is the worst direction and we know all components of $\boldsymbol{\mu}$ except the $k$th one, then by analyzing the reduced one-dimensional problem, we find the same rate as well since dimension disappears from the rate. Clearly this is oracle and unbeatable and so the rate is the best possible under anisotropic H\"{o}lder spaces. In the isotropic case where $\alpha_k=\alpha,k=1,\dotsc,d$, the second stage rate for $\boldsymbol{\mu}$ reduces to $n^{-(\alpha-1)/(2\alpha)}$ and this is precisely the minimax rate for this problem under sequential sampling (see \citet{max1988,tsybakov1990b} and \citet{2stage} as mentioned in the introduction).

\begin{remark}\label{rem:credible}\rm
There are other ways to construct credible set and obtain the second stage samples. A first attempt would be to simulate $f$ from its first stage posterior and apply the argmax operator on $f$. However as the posterior of $f$ contracts to $f_0$, this approach does not ensure that the second stage samples are sufficiently spread out, i.e., the distance between samples at direction $k$ is at least some constant multiple of $\delta_{n,k}$ for $k=1,\dotsc,d$. Another way is to do uniform sampling on $\mathcal{C}_{\boldsymbol{\mu}}$ constructed in \eqref{eq:mu credible}, that is, we envelope $\mathcal{C}_{\boldsymbol{\mu}}$, which is possibly irregularly shaped, by the smallest hypercube, do uniform sampling on this cube and discard points that fall outside of $\mathcal{C}_{\boldsymbol{\mu}}$. By (ii) and (iii) of Theorem 5.1, $\underline{\mathcal{C}}_{\boldsymbol{\mu}}\subset\mathcal{C}_{\boldsymbol{\mu}}\subset\overline{\mathcal{C}}_{\boldsymbol{\mu}}$, and samples in $\overline{\mathcal{C}}_{\boldsymbol{\mu}}$ are proportional to those in $\underline{\mathcal{C}}_{\boldsymbol{\mu}}$ under uniform sampling. Hence, the entries of $(\boldsymbol{Z}^T\boldsymbol{Z})^{-1}$ arising from uniform sampling on $\mathcal{C}_{\boldsymbol{\mu}}$ or on hypercubes will have the same order, and the second stage posteriors from these two sampling schemes will have the same asymptotic behavior. Operationally, sampling from $\mathcal{C}_{\boldsymbol{\mu}}$ requires an extra step in deciding whether the sampled points fall in the set or not.
\end{remark}

\begin{remark}\label{rem:prior}\rm
For asymptotic analyses, the prior covariance matrix $\boldsymbol{V}$ plays a minor role as the data ``washes" out the prior. However, for finite samples, the correct specification of $\boldsymbol{V}$ is crucial for the success of our proposed method in practical applications. Through empirical experiments, we discovered that $\boldsymbol{V}$ must reflect the scaling of the space by $\delta_{n,k}$ at direction $k$, and its inverse must have the same structure as $\boldsymbol{Z}^T\boldsymbol{Z}$ so that $(\boldsymbol{Z}^T\boldsymbol{Z}+\boldsymbol{V}^{-1})^{-1}$ will act as an effective shrinkage factor in \eqref{eq:ptheta}. Let us write $\boldsymbol{\Delta}:=\mathrm{diag}\left\{\prod_{k=1}^d\delta_{n,k}^{-(\boldsymbol{i}_j)_k}:j=0,1\ldots,W\right\}$ and we can see that $\boldsymbol{Z}^T\boldsymbol{Z}=\boldsymbol{\Delta A\Delta}$ where $\boldsymbol{A}$ is a matrix of constants not depending on $n$ (see Lemma \ref{lem:random} for more details). Therefore if we center the second stage design points, then we match the structure of $\boldsymbol{Z}^T\boldsymbol{Z}$ by choosing $\boldsymbol{V}=\boldsymbol{\Delta}^2$, which is our default choice. If the points are not centered, we found that Zellner's $g$-prior will work, i.e., $\boldsymbol{V}=g(\boldsymbol{Z}^T\boldsymbol{Z})^{-1}$ where $g$ can be estimated by empirical or hierarchical Bayes methods.
\end{remark}

\section{Simulation study}\label{sec:sim}
We shall compare the performance of our two-stage Bayesian procedure with two other estimation methods: the single-stage Bayesian, and the two-stage frequentist procedure proposed in \citet{2stage}. Consider the following true regression function defined on $[0,1]^2$:
\begin{align*}
f_0(x,y)&=\left(1+e^{-5(2x-1)^2-2(2y-1)^4}\right)[\cos{4(2x-1)}+\cos{5(2y-1)}]
\end{align*}
where the true mode is given by $\boldsymbol{\mu}_0=(0.5,0.5)^T$. In the first stage, we observe $f_0$ on a uniform $30\times30$ grid with i.i.d.~errors distributed as $\mathrm{N}(0,0.01)$ (see Figure \ref{fig:f0xyobs} with black circles as observations). We use bivariate tensor-product of B-splines with normal coefficients as our prior (see Section \ref{sec:tprior}). We choose the pair $(J_1,J_2)$ that maximizes its posterior, i.e., $$\Pi(J_1=j_1,J_2=j_2|\boldsymbol{Y},\sigma=\widetilde{\sigma}_1)\propto\widetilde{\sigma}_1^{-n_1}[\det{(\boldsymbol{B\Omega B}^T+\boldsymbol{I}_{n_1})}]^{-1/2}\Pi(J_1=j_1,J_2=j_2)$$ by integrating out $\boldsymbol{\theta}$. We create a candidate set $\mathcal{J}:=\{1,2,\dotsc,J_{\mathrm{max}}\}\times\{1,\dotsc,J_{\mathrm{max}}\}$ by setting $J_{\mathrm{max}}=20$. We put discrete uniform prior on $(J_1,J_2)$ over $\mathcal{J}$ such that $\Pi(J_1=j_1,J_2=j_2)=J_{\mathrm{max}}^{-2}=1/400$ for any $(j_1,j_2)\in\mathcal{J}$. We then find the combination that gives the maximum $\log{\Pi(J_1=j_1,J_2=j_2|\boldsymbol{Y},\sigma=\widetilde{\sigma}_1)}$ by doing a grid search. To speed up computations, we ignore any constant terms such as the prior factor and the posterior denominator since they do not affect this optimization problem. We plot this marginal log-posterior of $(J_1,J_2)$ in Figure \ref{fig:Jmax} and we found that $J_1=7$ and $J_2=9$ based on this criterion. At each dimension, the B-spline is of order 4 (cubic) with different uniform knot sequence. For the prior parameters, we set $\boldsymbol{\eta}=\boldsymbol{0}$ and $\boldsymbol{\Omega}=\boldsymbol{I}$. Figure \ref{fig:2points} shows the surface of the first stage posterior mean of $f$.

\begin{figure}
\centering
\begin{subfigure}[b]{0.48\textwidth}
\includegraphics[width=\textwidth]{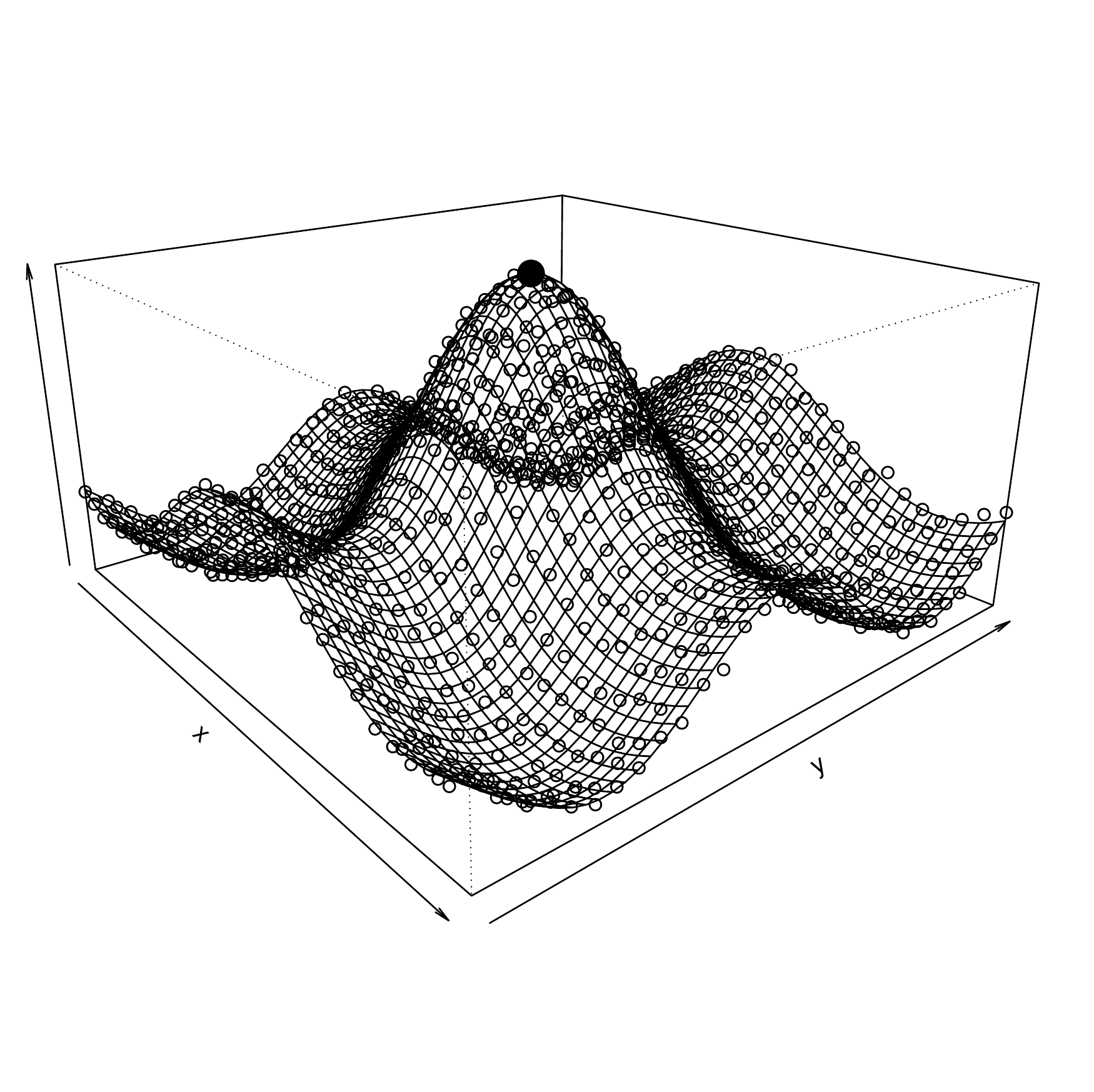}
\caption{A plot of $f_0$, with the black solid point as $f_0(\boldsymbol{\mu}_0)=M_0$ and black circles as the first stage $900$ observations.}
\label{fig:f0xyobs}
\end{subfigure}\quad
\begin{subfigure}[b]{0.48\textwidth}
\includegraphics[width=\textwidth]{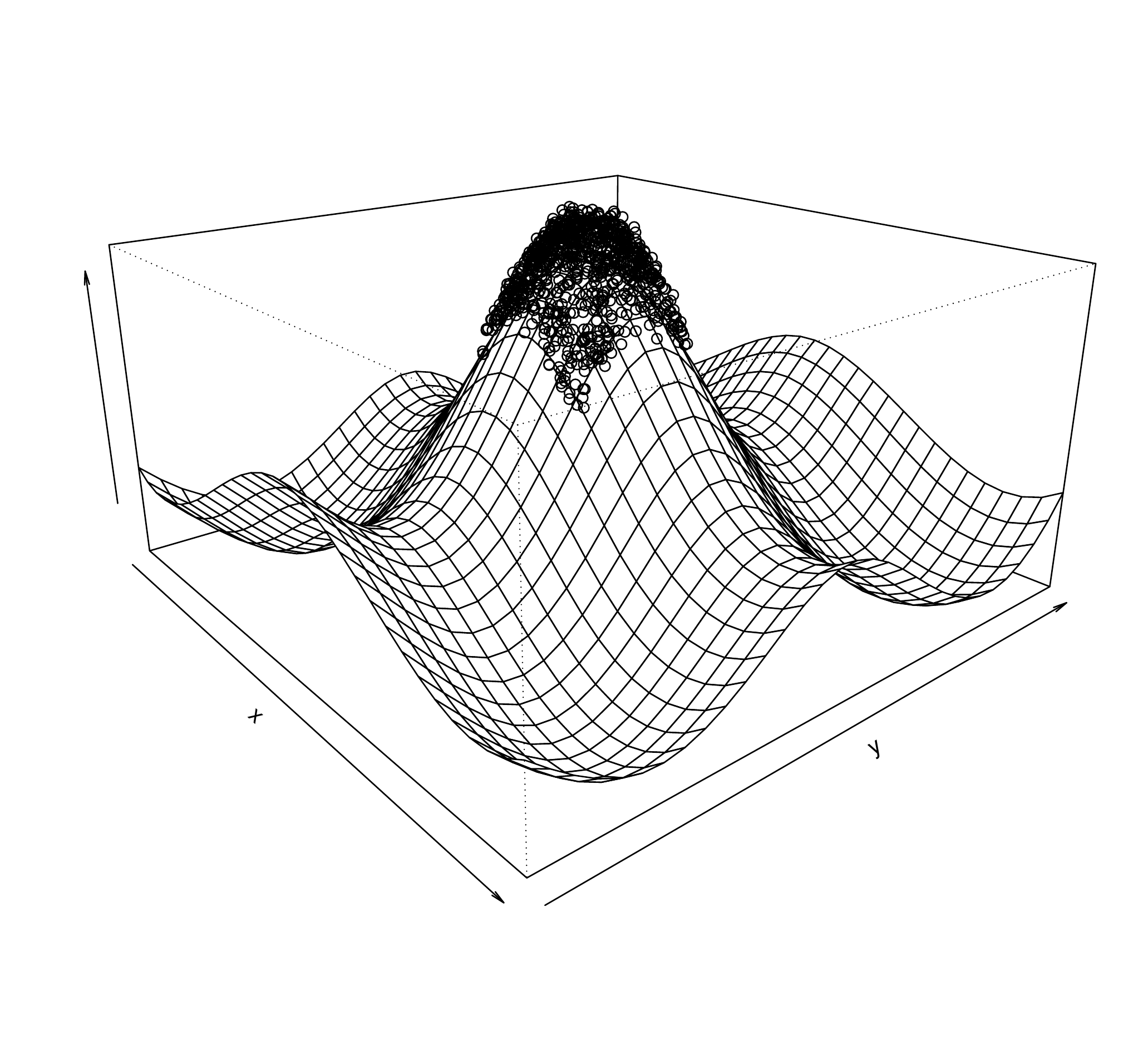}
\caption{Posterior mean based on bivariate B-spline prior, and the black circles are the $864$ second stage samples.}
\label{fig:2points}
\end{subfigure}
\caption{Our proposed Bayesian two-stage procedure: first stage on the left and second stage on the right.}
\end{figure}

\begin{figure}
\centerline{\includegraphics[scale=0.45]{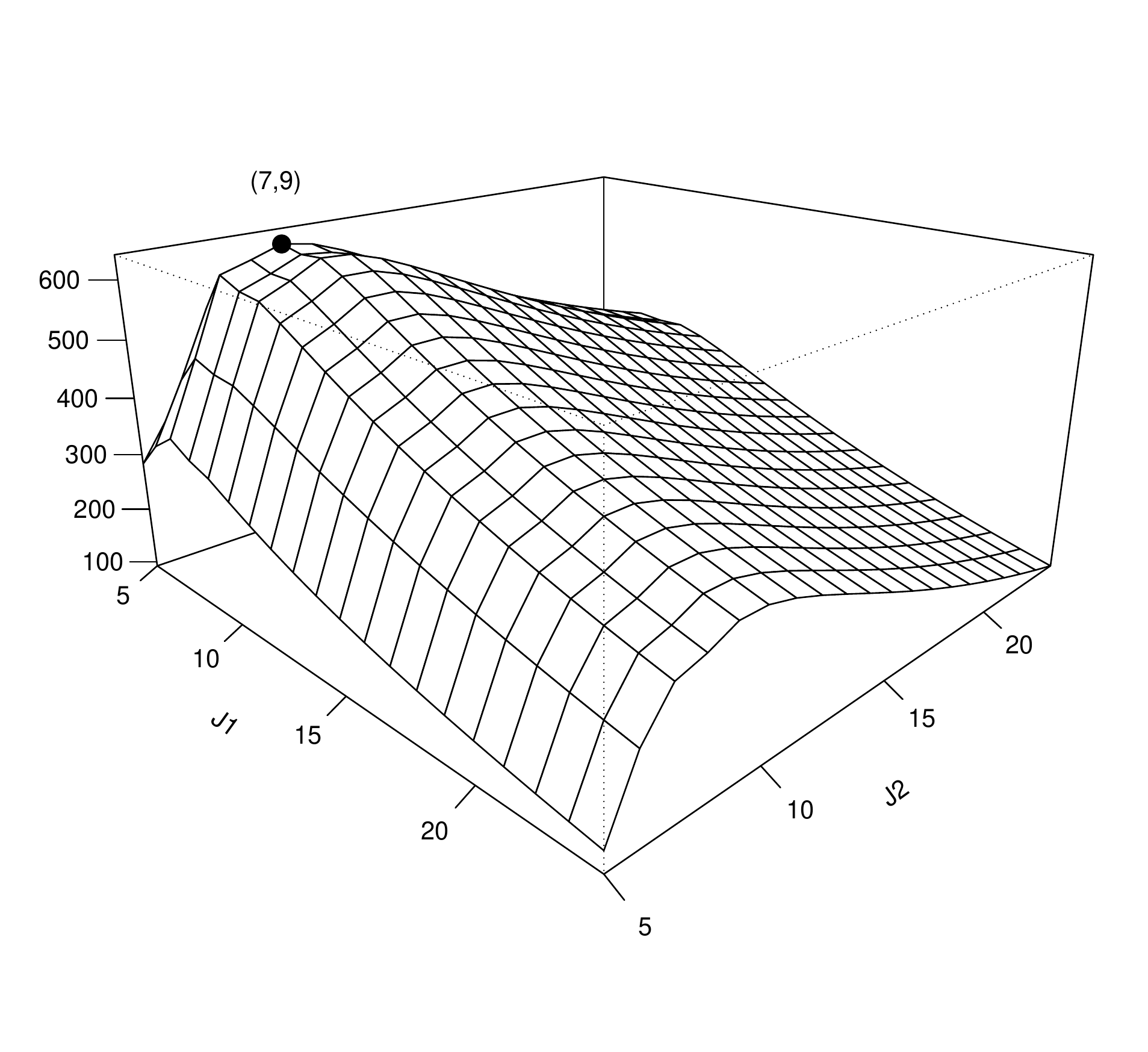}}
\vspace{-20pt}
\caption{log-posterior of $(J_1,J_2)$ with its maximum at $(7,9)$.}
\label{fig:Jmax}
\end{figure}

We sample 864 points uniformly in $\{\boldsymbol{\mu}:|\mu_k-\widetilde{\mu}_k|\leq\delta_k, k=1,2\}$ (see black circles in Figure \ref{fig:2points}) to obtain the second stage data $Y_i,i=1,\dotsc,864$. This number is chosen so as to fulfill the sampling configuration required by the frequentist method explained below, and to ensure that all methods under consideration will use the same amount of samples. To choose $\delta_1,\delta_2$, we first draw $1000$ samples from the first stage posterior distribution of $f$, find the mode for each sample by grid search to yield samples from the first stage posterior of $\boldsymbol{\mu}$, search for the smallest rectangle enveloping these induced $\boldsymbol{\mu}$ samples, and take $\delta_1,\delta_2$ be the lengths of its sides. We found that $\delta_1=0.1111$ and $\delta_2=0.1111$ and we proceed to do uniform sampling across this rectangular region. Note that we do not take the induced $\boldsymbol{\mu}$ samples as our second stage samples because most of them are concentrated near the center, and hence they are not sufficiently spread out (see Remark \ref{rem:credible}).

We center the $864$ sampled design points at the origin by subtracting each of them by $\widetilde{\boldsymbol{\mu}}$, and we use tensor product of quadratic polynomials with normal coefficients as prior (see \eqref{eq:thetaf}). That is for $x,y\in\mathcal{Q}=[-\delta_1,\delta_1]\times [-\delta_2,\delta_2]$,  $f_{\boldsymbol{\theta}}(x,y)=\theta_0+\theta_1x+\theta_2y+\theta_3x^2+\theta_4y^2+\theta_5xy+\theta_6xy^2+\theta_7x^2y+\theta_8x^2y^2$. Since $x,y\in \mathcal{Q}$, we note that the last three columns of the constructed basis matrix $\boldsymbol{Z}$ (corresponding to the last three terms) are very small in magnitude compared with the remaining terms. Thus for numerical simplicity, we consider only
\begin{align*}
f^{*}_{\boldsymbol{\theta}}(x,y)=\theta_0+\theta_1x+\theta_2y+\theta_3x^2+\theta_4y^2+\theta_5xy,
\end{align*}
where $\boldsymbol{\theta}|\sigma^2\sim\mathrm{N}(\boldsymbol{0},\sigma^2\boldsymbol{V})$, with $\boldsymbol{V}=\mathrm{diag}(1,\delta_1^{-2},\delta_2^{-2},\delta_1^{-4},\delta_2^{-4},\delta_1^{-2}\delta_2^{-2})$. We use the empirical Bayes method to estimate $\sigma$ by $\widetilde{\sigma}_{*}$, which is the weighted average of the first and second stage estimates $\widetilde{\sigma}_1,\widetilde{\sigma}_2$. However, we note that in our simulations that $\widetilde{\sigma}_2$ gives a much better estimate than $\widetilde{\sigma}_1$ at the current $(n_1,n_2)$-sampling plan, and since both are valid independent estimates of $\sigma$, we take $\widetilde{\sigma}_{*}=\widetilde{\sigma}_2$. Now, to compute $\boldsymbol{\mu}_{\boldsymbol{z}}=\argmax_{\boldsymbol{z}\in\mathcal{Q}}f^{*}_{\boldsymbol{\theta}}(\boldsymbol{z})$ for a fixed $\boldsymbol{\theta}$, we solve the following system of equation $\nabla f^{*}_{\boldsymbol{\theta}}(\boldsymbol{\mu}_{\boldsymbol{z}})=\boldsymbol{0}$, which is equivalent to solving
\begin{align}\label{eq:2muind}
\begin{pmatrix}
2\theta_3&\theta_5\\
\theta_5 & 2\theta_4
\end{pmatrix}
\begin{pmatrix}
\mu_{\boldsymbol{z},1}\\ \mu_{\boldsymbol{z},2}
\end{pmatrix}
=\begin{pmatrix}
-\theta_1\\
-\theta_2
\end{pmatrix}
\end{align}
for $\boldsymbol{\mu}_{\boldsymbol{z}}=(\mu_{\boldsymbol{z},1},\mu_{\boldsymbol{z},2})^T$. Therefore, to induce the posterior distribution of $\boldsymbol{\mu}$, we draw samples from $\Pi(\boldsymbol{\theta}|\boldsymbol{Y})$ by substituting $\sigma=\widetilde{\sigma}_2$ in \eqref{eq:ptheta}, solving for $\boldsymbol{\mu}_{\boldsymbol{z}}$ using \eqref{eq:2muind} for each sample, and translating back to $\boldsymbol{\mu}=\widetilde{\boldsymbol{\mu}}+\boldsymbol{\mu}_{\boldsymbol{z}}$.

\begin{figure}[ht!]
\centerline{\includegraphics[scale=0.5]{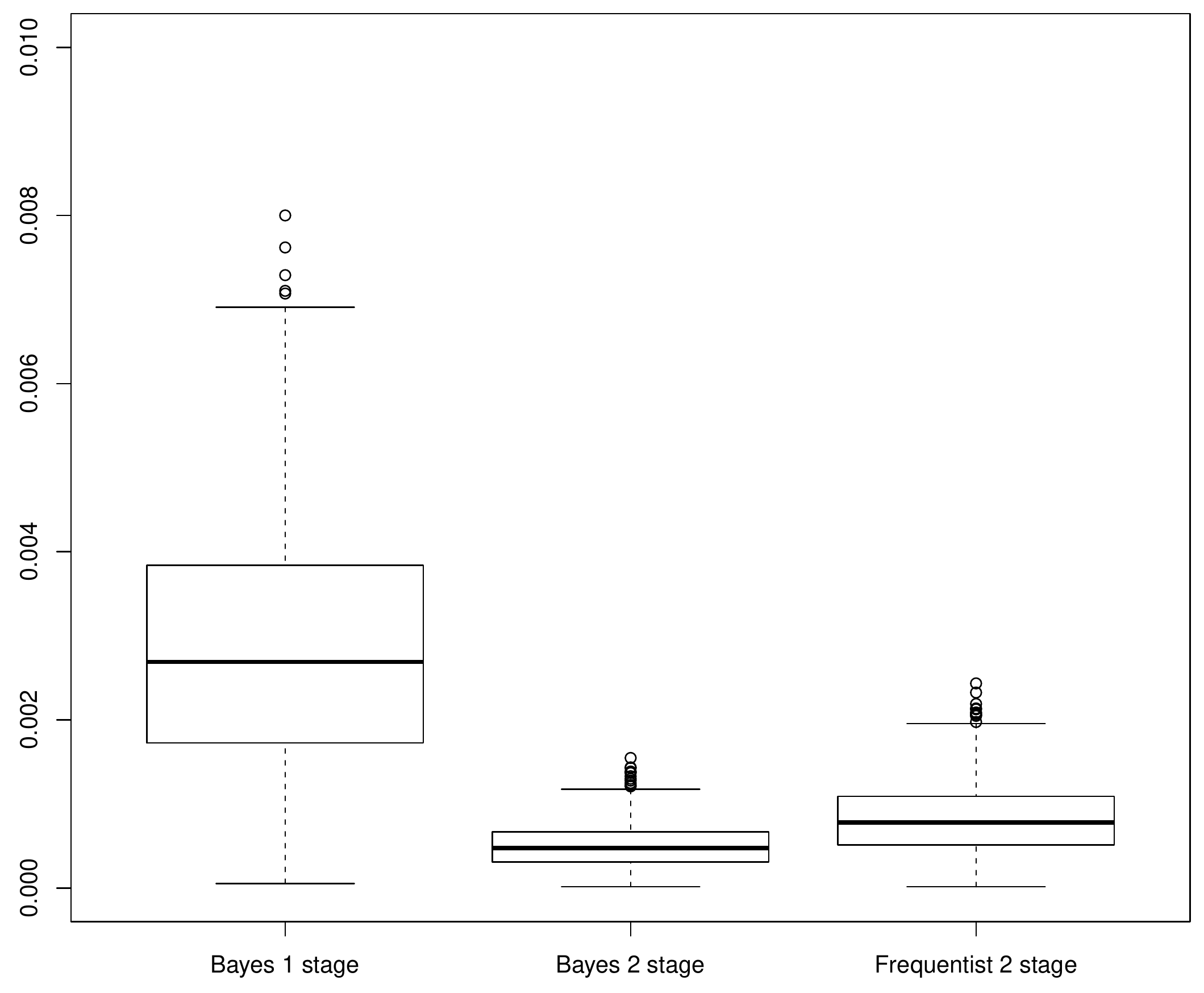}}
\caption{Root mean square errors ($y$-axis) for the Bayesian and frequentist procedures with $1000$ Monte Carlo replications.}
\vspace{-15pt}
\label{fig:result}
\end{figure}

The procedure described is implemented in the statistical software package \texttt{R}. Univariate B-splines are constructed using the \texttt{bs} function from the \texttt{splines} package. We then use \texttt{tensor.prod.model.matrix} from the \texttt{crs} package to form their tensor products. Observe that we have used a total of $30\times30+864=1764$ observations. To show that the two-stage procedure indeed has better accuracy than single stage procedures, we then compare our two-stage procedure with a Bayesian single stage method that uses the same number of samples, i.e., on a uniform $42\times42$ grid points. As in the previous case, we observe $f_0$ at these points with i.i.d.~errors $\mathrm{N}(0,0.01)$. We then use bivariate tensor-product B-splines with normal coefficients as prior with the same setting. We found that $J_1=9, J_2=9$ maximizes its posterior.

To compare Bayesian and frequentist procedures, we repeated the same experiment using the two-stage frequentist procedure implemented in \citet{2stage}. In their procedure, we first fit a locally linear surface by \texttt{loess} regression in \texttt{R} on the first stage design points. We choose the corresponding bandwidth or span parameter to be 0.02 by leave-one-out cross validation. The maximum of this surface serves as a preliminary estimator. We construct a rectangle of size $2\delta_1\times 2\delta_2$ around this estimator and further divide this region into 4 smaller $\delta_1\times \delta_2$ rectangles. In their implementation, \citet{2stage} actually tuned $\delta_1,\delta_2$ using the knowledge of the true maximum, and to the best of our knowledge, there is no practical frequentist method to choose the $\delta$s in the literature. Faced with this situation, we decided to follow \citet{2stage} and choose the $\delta$s by minimizing the expected $L_2$-distance between the second stage posterior mean for $\boldsymbol{\mu}$ and the true $\boldsymbol{\mu}_0$. We found that $\delta_1=0.06$ and $\delta_2=0.06$. Then, we take 96 replicated samples at the 9 grid points to form 864 second stage samples. A quadratic surface is fit through these points and its coefficients are estimated using least squares. We compute the second stage estimator of $\boldsymbol{\mu}_0$ by \eqref{eq:2muind} and call it $\widehat{\boldsymbol{\mu}}_2$.

Let $\widetilde{\boldsymbol{\mu}}_1$ and $\widetilde{\boldsymbol{\mu}}_2$ be the single stage and two-stage posterior means respectively. To compare the performance of our proposed Bayesian two-stage method with the other two, we replicate the experiment $1000$ times for each of the three methods. For each replicated experiment, we compute $\|\boldsymbol{\mu}-\boldsymbol{\mu}_0\|$ for each $\boldsymbol{\mu}=\widetilde{\boldsymbol{\mu}}_1,\widetilde{\boldsymbol{\mu}}_2,\widehat{\boldsymbol{\mu}}_2$. Figure \ref{fig:result} shows the box-plots of these $1000$ computed root mean square errors (RMSE) for all three procedures.

We see that the our proposed Bayesian two-stage procedure has considerably lower RMSE than the corresponding single stage procedure, and thus supports the conclusion of Theorem \ref{th:2M} in finite sample setting. We also observe that our proposed Bayesian two-stage procedure performs slightly better than the frequentist procedure, despite the fact that we have used the true maximum to tune the frequentist procedure while our proposed Bayesian two-stage method is fully data driven. The superior performance may be due to our choice of prior covariance matrix $\boldsymbol{V}$ (see Remark \ref{rem:prior}) where it causes $(\boldsymbol{Z}^T\boldsymbol{Z}+\boldsymbol{V}^{-1})^{-1}$ in the posterior (see \eqref{eq:ptheta}) to act as an effective shrinkage factor.

The \texttt{R} codes to reproduce all the results and figures in this section can be found in the first author's GitHub https://github.com/wwyoo/Bayes-two-stage.

\section{Conclusion and future works}\label{sec:conclude}
We studied Bayesian estimation of $\boldsymbol{\mu}$ and $M$ in two different settings. In the usual single stage situation, we obtained posterior contraction rates of $(\log{n}/n)^{\alpha^{*}\{1-\underline{\alpha}^{-1}\}/(2\alpha^{*}+d)}$ for $\boldsymbol{\mu}$ and $(\log{n}/n)^{\alpha^{*}/(2\alpha^{*}+d)}$ for $M$, under a tensor-product B-splines random series prior with Gaussian coefficients. Using our proposed Bayesian two-stage procedure, we can accelerate the aforementioned rates to $n^{-(\underline{\alpha}-1)/(2\underline{\alpha})}$ and $n^{-1/2}$ for $\boldsymbol{\mu}$ and $M$ respectively, as long as the optimal $\delta_{n,k}=n^{-1/(2\alpha_k)}$ is chosen as radius for the credible cube used in second stage sampling. This rate acceleration is remarkable because it removes the logarithmic factors and it mitigates the effect of dimension $d$ on the rates. We implemented a practical version of our Bayesian two-stage procedure in a simulation study, and it outperformed a traditional single stage Bayesian procedure by a large margin, and also performed slightly better compared to a frequentist procedure recently proposed in the literature.

An important future work for the single and two-stage Bayesian procedures is to make them adaptive to the unknown smoothness $\boldsymbol{\alpha}$. In other words, designing theoretically sound and data driven procedures to determine the optimal $J_k$ (number of B-splines) and $\delta_{n,k}$ (credible cube radius) for $k=1,\dotsc,d$. If only $L_2$-distances are studied, finite random series easily gives adaptation by simply putting a prior on $J_k$ the number of basis functions. For supremum $L_{\infty}$-distance as considered in this paper, getting adaptive posterior contraction rate is a lot more challenging and seems to need a different type of prior (see \citet{yoo2017}). Coverage of uniform norm credible sets in the adaptive setting may even need a more radical technique (cf.~\citet{yoo2018}).

For two-stage procedures, rate adaptation has not been yet possible even for the non-Bayesian procedure of \citet{2stage}. At the minimum, to adapt, one may need multi-stage (possibly more than 2) sampling. In our simulation study, the empirical Bayes step used in determining $J_k$ and the sampling based procedure to choose $\delta_1,\delta_2$ are practical and fast plug-in methods, but it is yet unknown how estimation uncertainty introduced by plugging-in these estimated quantities propagate throughout the two-stage procedure, and whether this is an optimal thing to do, even though they do give reasonable finite sample results. We shall try to answer these pressing questions in some future work.

\section*{Acknowledgement} We would like to thank the associate editor and the referees for their thought-provoking comments, as this led us to re-examine our results and improve our presentation.

\section{Proofs}\label{sec:proof}
Recall that the posterior mean of $D^{\boldsymbol{r}}f$ is  $D^{\boldsymbol{r}}\widetilde{f}:=\boldsymbol{A}_{\boldsymbol{r}}\boldsymbol{Y}+\boldsymbol{c}_{\boldsymbol{r}}\boldsymbol{\eta}$. Let us write the posterior contraction rate of $\boldsymbol{\mu}$ given in Theorem \ref{th:murate} as $\epsilon_n=\max_{1\leq k\leq d}\epsilon_{n,k}$, where $\epsilon_{n,k}=(\log{n}/n)^{\alpha^{*}(1-\alpha_k^{-1})/(2\alpha^{*}+d)}$.

\begin{proof}[Proof of Lemma \ref{lem:uniquemu}]
\noindent For the empirical Bayes case, the reproducing kernel Hilbert space of the Gaussian posterior process $\{f(\boldsymbol{t}):\boldsymbol{t}\in[0,1]^d\}$ given $\boldsymbol{Y}$ is the $J$-dimensional space of polynomial splines spanned by elements of $\boldsymbol{b}_{\boldsymbol{J},\boldsymbol{q}}(\cdot)$ (see Theorem 4.2 of \citet{van2008reproducing}). This implies that it has continuous sample path at every realization. Observe that $\widetilde{\sigma}_n^2>0$ almost surely. Then if
\begin{align*}
0&=\mathrm{Var}(f(\boldsymbol{t})-f(\boldsymbol{s})|\boldsymbol{Y},\sigma=\widetilde{\sigma}_n)\\
&=\widetilde{\sigma}_n^2(\boldsymbol{b}_{\boldsymbol{J},\boldsymbol{q}}(\boldsymbol{t})
-\boldsymbol{b}_{\boldsymbol{J},\boldsymbol{q}}(\boldsymbol{s}))^T\left(\boldsymbol{B}^T\boldsymbol{B}+\boldsymbol{\Omega}^{-1}\right)^{-1}
(\boldsymbol{b}_{\boldsymbol{J},\boldsymbol{q}}(\boldsymbol{t})-\boldsymbol{b}_{\boldsymbol{J},\boldsymbol{q}}(\boldsymbol{s})),
\end{align*}
this implies that $B_{j_k,q_k}(t_k)=B_{j_k,q_k}(s_k)$ for $j_k=1,\dotsc,J_k$, and $k=1,\dotsc,d$. Since $q_k\geq\alpha_k>2$ for $k=1,\dotsc,d$, this rules out the possibility that the B-splines are step functions and further implies that $\boldsymbol{t}=\boldsymbol{s}$. Thus by Lemma 2.6 of \citet{gaussianunique}, $\boldsymbol{\mu}$ is unique for almost all sample paths of $\{f(\boldsymbol{t}):\boldsymbol{t}\in[0,1]^d\}$ under the empirical posterior distribution.

For hierarchical Bayes, consider the conditional posterior process $\Pi(f|\boldsymbol{Y},\sigma)$ for arbitrary $\sigma>0$. The reproducing kernel Hilbert space of this Gaussian process is still the space of splines as before, and consequently has continuous sample path for each realization. Now by substituting $\sigma$ for $\widetilde{\sigma}_n$ in the preceding display, we see that $\mathrm{Var}(f(\boldsymbol{t})-f(\boldsymbol{s})|\boldsymbol{Y},\sigma)=0$ implies $\boldsymbol{t}=\boldsymbol{s}$. Again by Lemma 2.6 of \citet{gaussianunique}, almost every sample path of $f$ given $\sigma$ has unique maximum for any $\sigma>0$. Note that $f$ is generated from the following scheme: draw $\sigma\sim\Pi(\sigma|\boldsymbol{Y})$, and then $f|\sigma\sim\mathrm{GP}(\boldsymbol{AY}+\boldsymbol{c\eta},\sigma^2\boldsymbol{\Sigma})$. Hence almost every draw of $f$ has unique maximum $\boldsymbol{\mu}$ under $\Pi(f|\boldsymbol{Y})$.
\end{proof}

\begin{proof}[Proof of Theorem~\ref{th:murate}]
If both $f,f_0\geq0$, then \eqref{eq:maxinequality} follows from the reverse triangle inequality by noting that $M=\|f\|_\infty$ and $M_0=\|f_0\|_\infty$ in this case. If the condition fails to hold, we can add a large enough constant $C>0$ such that $g=f+C\geq0$ and $g_0=f_0+C\geq0$. Then by the reverse triangle inequality, $|\|g\|_\infty-\|g_0\|_\infty|\leq\|g-g_0\|_\infty$. The right hand side is $\|f-f_0\|_\infty$, while the left hand side is $|M-M_0|$ because $\|g\|_\infty=M+C,\|g_0\|_\infty=M_0+C$.

To prove \eqref{eq:muinequality}, we need to first establish consistency of the induced posterior of $\bm{\mu}$. Now for any $\epsilon>0$ and $\delta>0$, $\Pi\left(\|\bm{\mu}-\bm{\mu}_0\|>\epsilon\middle|\bm{Y}\right)$ is bounded above by
\begin{align*}
&\Pi\left(\sup_{\boldsymbol{x}\notin\mathcal{B}(\boldsymbol{\mu}_0,\epsilon)} f(\bx)>f(\bm{\mu}_0)\middle|\bm{Y}\right)\\
&\leq\Pi\left(\sup_{\boldsymbol{x}\notin\mathcal{B}(\boldsymbol{\mu}_0,\epsilon)}f(\bx)>f(\bm{\mu}_0),\quad f(\bm{\mu}_0)\geq f_0(\bm{\mu}_0)-\delta/2\middle|\bm{Y}\right)\\
&\qquad+\Pi\left(f(\bm{\mu}_0)< f_0(\bm{\mu}_0)-\delta/2\middle|\bm{Y}\right).
\end{align*}
The second term is bounded above by $\Pi(|f(\boldsymbol{\mu}_0)-f_0(\boldsymbol{\mu}_0)|>\delta/2|\boldsymbol{Y})\leq\Pi(\|f-f_0\|_\infty>\delta/2|\boldsymbol{Y})$, and this goes to 0 in $P_0$-probability by Theorem \ref{th:frate} in Section \ref{sec:appendix} Appendix with $\boldsymbol{r}=\boldsymbol{0}$. The well-separation property of Assumption 2 implies that there exists a $\delta>0$, such that $f_0(\bx)<f_0(\bm{\mu}_0)-\delta$ for $\boldsymbol{x}\notin\mathcal{B}(\boldsymbol{\mu}_0,\epsilon)$. Hence, for this $\delta>0$ and appealing again to Theorem \ref{th:frate}, the first term is bounded above by
$$\Pi\left(\bigcup_{\boldsymbol{x}\notin\mathcal{B}(\boldsymbol{\mu}_0,\epsilon)} \left \{f(\bx)>f_0(\bx)+\frac{\delta}{2}\right\} \middle|\bm{Y}\right)
\leq\Pi\left(\|f-f_0\|_\infty>\frac{\delta}{2}\middle|\bm{Y}\right)\rightarrow0$$
in $P_0$-probability as $n\rightarrow\infty$. Thus the posterior distribution of $\bm{\mu}$ is consistent at $\bm{\mu}_0$.

Now by Taylor's theorem,
$\nabla f_0(\boldsymbol{\mu})=\nabla f_0(\boldsymbol{\mu}_0)+\boldsymbol{H}f_0(\boldsymbol{\mu}^{*})(\boldsymbol{\mu}-\boldsymbol{\mu}_0)$
for some $\boldsymbol{\mu}^{*}=\lambda\boldsymbol{\mu}+(1-\lambda)\boldsymbol{\mu}_0$ with $\lambda\in(0,1)$. Since the posterior of $\boldsymbol{\mu}$ is consistent as shown above and $\boldsymbol{\mu}^{*}$ falls in between $\boldsymbol{\mu}$ and $\boldsymbol{\mu}_0$, it must be that as $n\rightarrow\infty$ and for any $\epsilon>0$, $\Pi(\|\boldsymbol{\mu}^{*}-\boldsymbol{\mu}_0\|\leq\epsilon|\boldsymbol{Y})\overset{P_0}{\longrightarrow}1$. Let us introduce the set $\mathcal{B}=\{\lambda_\text{max}[\boldsymbol{H}f_0(\boldsymbol{\mu}^{*})]<-\lambda_0\}$, and note that under Assumptions 3, $\Pi(\mathcal{B}|\boldsymbol{Y})\overset{P_0}{\longrightarrow}1$ and $\boldsymbol{H}f_0(\boldsymbol{\mu}^{*})$ is invertible with posterior probability tending to one. Therefore, as $n\rightarrow\infty$ and intersecting with $\mathcal{B}$,
$\boldsymbol{\mu}-\boldsymbol{\mu}_0=\boldsymbol{H}f_0(\boldsymbol{\mu}^{*})^{-1}(\nabla f_0(\boldsymbol{\mu})-\nabla f_0(\boldsymbol{\mu}_0)).$
Noting that $\nabla f_0(\boldsymbol{\mu}_0)=\nabla f(\boldsymbol{\mu})=\boldsymbol{0}$ by Assumption 2, then
\begin{align*}
\|\boldsymbol{\mu}-\boldsymbol{\mu}_0\|\leq\frac{1}{\lambda_0}\|\nabla f_0(\boldsymbol{\mu})-\nabla f(\boldsymbol{\mu})\|\leq\frac{\sqrt{d}}{\lambda_0}\max_{1\leq k\leq d}\|D_kf-D_kf_0\|_\infty,
\end{align*}
where we have used the sub-multiplicative property of the $\|\cdot\|_{(2,2)}$-norm, i.e., $\|\boldsymbol{Ay}\|\leq\|\boldsymbol{A}\|_{(2,2)}\|\boldsymbol{y}\|$ for some matrix $\boldsymbol{A}$ and vector $\boldsymbol{y}$.

Theorem \ref{th:frate} with $\boldsymbol{r}=\boldsymbol{0}$ together with \eqref{eq:maxinequality} now proves the desired contraction rate on $M$. To derive the rate \eqref{eq:rate mu} for $\bm{\mu}$, apply again Theorem \ref{th:frate} with $\boldsymbol{r}=\boldsymbol{e}_k$, and the $L_\infty$-contraction rate for $D_kf$ is $\epsilon_{n,k}, k=1,\dotsc,d$. In view of \eqref{eq:muinequality}, we have for any $m_n\rightarrow\infty$,
$$
\mathrm{E}_0\Pi(\|\boldsymbol{\mu}-\boldsymbol{\mu}_0\|>m_n\epsilon_n|\boldsymbol{Y})
\le \sum_{k=1}^d\mathrm{E}_0\Pi\left(\|D_kf-D_kf_0\|_\infty>\frac{\lambda_0}{\sqrt{d}}m_n\epsilon_n\middle|\boldsymbol{Y}\right)$$
approaches $0$ uniformly in $\|f_0\|_{\boldsymbol{\alpha},\infty}\leq R$, establishing the assertion.
\end{proof}

\begin{proof}[Proof of Corollary 4.2]
From the proof of Theorem 4.1 before, we know that $\boldsymbol{\mu}-\boldsymbol{\mu}_0=\boldsymbol{H}f_0(\boldsymbol{\mu}^{*})^{-1}(\nabla f_0(\boldsymbol{\mu})-\nabla f_0(\boldsymbol{\mu}_0))$. Noting that $\nabla f_0(\boldsymbol{\mu}_0)=\nabla f(\boldsymbol{\mu})=\boldsymbol{0}$ by Assumption 2, we can use the fact $\|\boldsymbol{Ab}\|^2\geq\lambda_{\mathrm{min}}(\boldsymbol{A}^T\boldsymbol{A})\|\boldsymbol{b}\|^2$ to write
\begin{align*}
\|\boldsymbol{\mu}-\boldsymbol{\mu}_0\|&\geq\sqrt{\lambda_{\mathrm{max}}^{-2}\{\boldsymbol{H}f_0(\boldsymbol{\mu}^{*})\}}\|\nabla f_0(\boldsymbol{\mu})-\nabla f_0(\boldsymbol{\mu}_0)\|\\
&\geq\lambda_1^{-1}\|\nabla f_0(\boldsymbol{\mu})-\nabla f(\boldsymbol{\mu})\|,
\end{align*}
by posterior consistency of $\boldsymbol{\mu}^{*}$ as established in the proof of Theorem 5.2. Let $\delta_n\to0$ be some sequence. Then for some small enough constant $h>0$ to be determined below, we have
\begin{align*}
\Pi(\|\boldsymbol{\mu}-\boldsymbol{\mu}_0\|\leq h\epsilon_n|\boldsymbol{Y})&\leq\Pi\left(\|\nabla f_0(\boldsymbol{\mu})-\nabla f(\boldsymbol{\mu})\|\leq\lambda_1h\epsilon_n,\|\boldsymbol{\mu}-\boldsymbol{\mu}_0\|\leq\delta_n\middle|\boldsymbol{Y}\right)\\
&\qquad+\Pi(\|\boldsymbol{\mu}-\boldsymbol{\mu}_0\|>\delta_n|\boldsymbol{Y}).
\end{align*}
Since the posterior of $\boldsymbol{\mu}$ is consistent, the second term is $o_{P_0}(1)$. Using the definition of continuity of $\boldsymbol{x}\mapsto\|\nabla f_0(\boldsymbol{x})-\nabla f(\boldsymbol{x})\|$ at $\boldsymbol{\mu}_0$ and by taking $n$ large enough (so that $\delta_n$ is small enough), we see that
\begin{align*}
\Pi(\|\boldsymbol{\mu}-\boldsymbol{\mu}_0\|\leq h\epsilon_n|\boldsymbol{Y})\leq\Pi\left[\|\nabla f_0(\boldsymbol{\mu}_0)-\nabla f(\boldsymbol{\mu}_0)\|\leq2\lambda_1h\epsilon_n\middle|\boldsymbol{Y}\right]+o_{P_0}(1).
\end{align*}
To obtain the same rate as the upper bound presented in (4.3) of Theorem 4.1, we then need the lower bound point-wise version of Theorem 9.1, namely for some constant $m_0>0$ and for any $\boldsymbol{x}\in[0,1]^d$,
\begin{align}\label{eq:flow}
\sup_{\|f_0\|_{\boldsymbol{\alpha},\infty}\leq R}\mathrm{E}_0\Pi\left(|D^{\boldsymbol{r}}f(\boldsymbol{x})-D^{\boldsymbol{r}}f_0(\boldsymbol{x})|\leq m_0n^{-\alpha^{*}\{1-\sum_{k=1}^d(r_k/\alpha_k)\}/(2\alpha^{*}+d)}\middle|\boldsymbol{Y}\right)\rightarrow0.
\end{align}
One can proceed to establish such lower bound directly since we have analytical expression for the Gaussian posterior distribution. By taking $\boldsymbol{r}=\boldsymbol{e}_k$ and $h\leq m_0/(2\lambda_1)$, we conclude that $\epsilon_n^2=\sum_{k=1}^dn^{-2\alpha^{*}(1-\alpha_k^{-1})/(2\alpha^{*}+d)}\geq\max_{1\leq k\leq d}n^{-2\alpha^{*}(1-\alpha_k^{-1})/(2\alpha^{*}+d)}$. As a result, if one adds an extra lower bound assumption (4.5), we have the lower bound:
\begin{align*}
\mathrm{E}_0\Pi\left(\|\boldsymbol{\mu}-\boldsymbol{\mu}_0\|\geq hn^{-\alpha^{*}\{1-(\min_{1\leq k\leq d}\alpha_k)^{-1}\}/(2\alpha^{*}+d)}\middle|\boldsymbol{Y}\right)\rightarrow1,
\end{align*}
for a small enough constant $h>0$. For the posterior lower bound of $M$, let $\boldsymbol{\mu}^{*}$ be some point in between $\boldsymbol{\mu}$ and $\boldsymbol{\mu}_0$. We Taylor expand $f_0$ around $\boldsymbol{\mu}_0$, add and subtract $M$ and use the reverse triangle inequality to write
\begin{align*}
|M_0-M|&\geq|f_0(\boldsymbol{\mu})-f(\boldsymbol{\mu})|+0.5(\boldsymbol{\mu}-\boldsymbol{\mu}_0)^T\boldsymbol{H}f_0(\boldsymbol{\mu}^{*})
(\boldsymbol{\mu}-\boldsymbol{\mu}_0)\\
&\geq|f_0(\boldsymbol{\mu})-f(\boldsymbol{\mu})|-0.5\lambda_1\|\boldsymbol{\mu}-\boldsymbol{\mu}_0\|^2,
\end{align*}
by the extra assumption and posterior consistency of $\boldsymbol{\mu}^{*}$. Choose $m_n=\sqrt{\log\log{n}}$ and define the set $\mathcal{T}:=\{\|\boldsymbol{\mu}-\boldsymbol{\mu}_0\|\leq m_n\epsilon_n\}$. Then for $\omega_n:=n^{-\alpha^{*}/(2\alpha^{*}+d)}$ and a small enough constant $h>0$ to be determined below,
\begin{align*}
\Pi(|M_0-M|\leq h\omega_n|\boldsymbol{Y})&\leq\Pi\left(|f_0(\boldsymbol{\mu})-f(\boldsymbol{\mu})|-0.5\lambda_1
\|\boldsymbol{\mu}-\boldsymbol{\mu}_0\|^2\leq h\omega_n,\mathcal{T}|\boldsymbol{Y}\right)\\
&\qquad+\Pi(\mathcal{T}^c|\boldsymbol{Y})\\
&\leq\Pi(|f_0(\boldsymbol{\mu})-f(\boldsymbol{\mu})|\leq h\omega_n+0.5\lambda_1m_n^2\epsilon_n^2|\boldsymbol{Y})+o_{P_0}(1),
\end{align*}
where the last term follows from (4.3) of Theorem 4.1. Using the continuity argument as before for $\boldsymbol{x}\mapsto|f_0(\boldsymbol{x})-f(\boldsymbol{x})|$ and the fact that $h\omega_n\gg\lambda_1m_n^2\epsilon_n^2$ when $\min_{1\leq k\leq d}\alpha_k>2$, we can further bound the right hand side above by
\begin{align*}
\Pi(|f_0(\boldsymbol{\mu})-f(\boldsymbol{\mu})|\leq2h\omega_n|\boldsymbol{Y})+o_{P_0}(1),
\end{align*}
for large enough $n$. By setting $\boldsymbol{r}=\boldsymbol{0}$ in \eqref{eq:flow} above, we conclude that the first term is $o_{P_0}(1)$ when $h\leq m_0/2$ and the second posterior statement on $M$ is established.
\end{proof}

\begin{proof}[Proof of Theorem \ref{th:crmurho}]
By construction, $\mathcal{C}_{\bm{\mu}}$ contains $\widetilde{\bm{\mu}}$, the mode of the posterior mean $\widetilde{f}$ of $f$. Since $\gamma<1/2$, $R_{n,k,\gamma}$ is greater than the posterior median of $\|D_k f- D_k \widetilde f\|_\infty$. For the empirical Bayesian posterior, $D_k f- D_k\widetilde f$ is a Gaussian process under Assumption 1 and in view of the second assertion of Theorem \ref{th:fcred} in the Appendix, it follows that $R_{n,k,\gamma}\gtrsim\epsilon_{n,k}$ (with $\boldsymbol{r}=\boldsymbol{e}_k$). Define $\nu_{n,k}^2:=\sup_{\bm{x}\in[0,1]^d} \mathrm{var} (D_k f (\bm{x})- D_k \widetilde f(\bm{x})|\bm{Y})$. Since the empirical Bayes estimate $\widetilde{\sigma}_n^2$ is consistent in view of (a) in Proposition \ref{prop:var1}, then by applying the inequality $\boldsymbol{y}^T\boldsymbol{Ay}\leq\lambda_{\mathrm{max}}(\boldsymbol{A})\|\boldsymbol{y}\|^2$ for any square matrix $\boldsymbol{A}$, we can bound $\mathrm{var} (D_k f (\bm{x})- D_k \widetilde f(\bm{x})|\bm{Y})$ with expression given in \eqref{eq:tsigmar} by
\begin{align*}
&(\sigma_0^2+o_{P_0}(1))\lambda_{\mathrm{min}}^{-1}(\boldsymbol{B}^T\boldsymbol{B}+\boldsymbol{\Omega}^{-1})
\lambda_{\mathrm{max}}(\boldsymbol{W}_{\boldsymbol{e}_k}\boldsymbol{W}_{\boldsymbol{e}_k}^T)\|\boldsymbol{b}_{\boldsymbol{J},\boldsymbol{q}-\boldsymbol{e}_k}(\boldsymbol{x})\|^2\\
&\qquad\lesssim n^{-1}J_k^2\prod_{l=1}^dJ_l\lesssim n^{-2\alpha^{*}(1-\alpha_k^{-1})/(2\alpha^{*}+d)},
\end{align*}
for any $\boldsymbol{x}\in[0,1]^d$. In the above, we have used the fact $0\leq B_{j_l,q_l}(x_l)\leq1$ and the partition of unity property of B-splines $\sum_{\boldsymbol{j}}\prod_{l=1}^dB_{j_l,q_l}(x_l)=1$ to bound $\|\boldsymbol{b}_{\boldsymbol{J},\boldsymbol{q}-\boldsymbol{e}_k}(\boldsymbol{x})\|^2\leq\sum_{j_1=1}^{J_1}\cdots\sum_{j_d=1}^{J_d}\prod_{l=1}^dB_{j_l,q_l-\mathbbm{1}_{\{l=k\}}}(x_l)\leq1$ for any $\boldsymbol{x}\in[0,1]^d$. The eigenvalues were bounded by \eqref{eq:BBO} of Lemma \ref{lem:BB} and Lemma \ref{lem:wr} (with $\boldsymbol{r}=\boldsymbol{e}_k$). Thus, we conclude that $\nu_{n,k}^2=o(\epsilon_{n,k}^2)$. Consequently by Borell's inequality (cf.~Proposition~A.2.1 of \citet{empirical}) and taking $\rho$ to be large enough, e.g. $\rho>1$,
\begin{eqnarray*}
\Pi (\bm{\mu} \not\in \mathcal{C}_{\bm{\mu}}|\bm{Y}) &\leq
\sum_{k=1}^d \Pi ( \|D_k f- D_k \widetilde f\|_\infty >\rho R_{n,k,\gamma}|\bm{Y})\\
&\leq d \max_{1\leq k\leq d}\exp [ -(\rho-1)^2 R_{n,k,\gamma}^2/(2\nu_{n,k}^2)]
\end{eqnarray*}
which goes to zero since $R_{n,k,\gamma}^2/\nu_{n,k}^2\rightarrow\infty$. Thus the credibility of $\mathcal{C}_{\bm{\mu}}$ tends to $1$ (or is at least $1-\gamma$) in $P_0$-probability as $n\rightarrow\infty$. For the hierarchical Bayesian posterior, the same conclusion follows since conditionally on $\sigma$, the posterior law obeys a Gaussian process and $\sigma$ lies in a small neighborhood of the true $\sigma_0$ with high posterior probability (from (c) of Proposition \ref{prop:var1}). To ensure coverage, note that by the construction of $\mathcal{C}_{\bm{\mu}}$, it contains $\bm{\mu}_0$ if $\|D_k f_0- D_k \widetilde f\|_\infty\le \rho R_{n,k,\gamma}$ for all $k=1,\ldots,d$. When $f_0$ is the true regression function, the $P_0$-probability of the last event tends to one for all $k$ by Theorem \ref{th:fcred} with $\boldsymbol{r}=\boldsymbol{e}_k$, and hence the statement on coverage is established. This proves assertion (i).

Assertion (ii) follows in view of \eqref{eq:mu credible} and if
\begin{align}\label{eq:mutilde}
\|\widetilde{\boldsymbol{\mu}}-\boldsymbol{\mu}_0\|\leq\frac{\sqrt{d}}{\lambda_0}\max_{1\leq k\leq d}\|D_k\widetilde{f}-D_kf_0\|_\infty
\end{align}
holds with $P_0$-probability tending to $1$. Indeed by Remark \ref{rem:uniquemutilde}, $\nabla\widetilde{f}(\widetilde{\boldsymbol{\mu}})=\boldsymbol{0}$ and the Hessian matrix $\boldsymbol{H}\widetilde{f}(\widetilde{\boldsymbol{\mu}})$ is non-negative definite and symmetric. Moreover, since $\widetilde{f}$ is a polynomial splines of order $q_k\geq\alpha_k>2,k=1,\dotsc,d$, by Assumption 2, $\boldsymbol{H}\widetilde{f}(\boldsymbol{x})$ is continuous. Now since $\widetilde{f}\rightarrow f_0$ uniformly in $P_0$-probability as a result of Theorem \ref{th:frate} and $f_0$ has a well-separated maximum by Assumption 2, it follows that $\widetilde{\boldsymbol{\mu}}$ is consistent in estimating $\boldsymbol{\mu}_0$ by Theorem 5.7 of \citet{asymp}. Hence for $n$ large enough, $\mathcal{B}(\widetilde{\boldsymbol{\mu}},\tau/2)$ is contained in $\mathcal{B}(\boldsymbol{\mu}_0,\tau)$ for the same $\tau$ appearing in Assumption 3. Now since $D^{\boldsymbol{r}}\widetilde{f}$ converges uniformly to $D^{\boldsymbol{r}}f_0$ (Theorem \ref{th:frate}), and using the fact that maximum eigenvalue is a continuous operation, we have by the continuous mapping theorem that $\lambda_{\mathrm{max}}(\boldsymbol{H}\widetilde{f}(\boldsymbol{x}))\rightarrow\lambda_{\mathrm{max}}(\boldsymbol{H}f_0(\boldsymbol{x}))$ uniformly in $\boldsymbol{x}$. By further adapting \eqref{eq:maxinequality} to the present situation, we see that $\sup_{\boldsymbol{x}\in\mathcal{B}(\widetilde{\boldsymbol{\mu}},\tau/2)}\lambda_{\mathrm{max}}(\boldsymbol{H}\widetilde{f}(\boldsymbol{x}))<-\lambda_0$ for sufficiently large $n$. Consequently, the proof of the inequality \eqref{eq:muinequality} given in Theorem \ref{th:murate} will go through even if we replace $f_0$ with $\widetilde{f}$ and $\boldsymbol{\mu}_0$ with $\widetilde{\boldsymbol{\mu}}$.

Let $\boldsymbol{1}_d$ be a $d$-dimensional vector of ones, and $\boldsymbol{\xi}$ some point between $\boldsymbol{x}$ and $\boldsymbol{x}-(Rd)^{-1}\boldsymbol{1}_d\rho R_{n,k,\gamma}$ for any given $\boldsymbol{x}\in[0,1]^d$. In what follows, we take $n$ large enough so that $R_{n,k,\gamma}$ is small enough, and adding or subtracting some constant multiple of $R_{n,k,\gamma}$ still allows us to stay within $[0,1]^d$. For (iii), we have by the multivariate mean value theorem and the Cauchy-Schwarz inequality that
\begin{align*}
\left|D_k\widetilde{f}(\boldsymbol{x}-(Rd)^{-1}\boldsymbol{1}_d\rho R_{n,k,\gamma})-D_k\widetilde{f}(\boldsymbol{x})\right|\leq\|\nabla D_k\widetilde{f}(\boldsymbol{\xi})\|R^{-1}d^{-1/2}\rho R_{n,k,\gamma},
\end{align*}
for $k=1,\dotsc,d$. Since $D^{\boldsymbol{r}}\widetilde{f}\rightarrow D^{\boldsymbol{r}}f_0$ uniformly (cf. Theorem \ref{th:frate}) and the norm is a continuous map, $\|\nabla D_k\widetilde{f}(\boldsymbol{\xi})\|\rightarrow\|\nabla D_kf_0(\boldsymbol{\xi})\|\leq\sqrt{d}\max_{1\leq j\leq d}|D_jD_kf_0(\boldsymbol{\xi})|\leq\sqrt{d}\|f_0\|_{\boldsymbol{\alpha},\infty}$ in view of the definition given in \eqref{eq:aninorm}. Therefore uniformly over $\|f_0\|_{\boldsymbol{\alpha},\infty}\leq R$, the right hand side above is less than $\rho R_{n,k,\gamma}$ when $n$ is large enough. Now since the mode of $\widetilde{f}(\cdot-(Rd)^{-1}\bm{1}_d\rho R_{n,k,\gamma})$ is $\widetilde{\boldsymbol{\mu}}+(Rd)^{-1}\bm{1}_d\rho R_{n,k,\gamma}$, it follows immediately that $\widetilde{\boldsymbol{\mu}}+(Rd)^{-1}\boldsymbol{1}_d\rho \max_{1\leq k\leq d}R_{n,k,\gamma}\in \mathcal{C}_{\bm{\mu}}$. Notice that this argument still holds even when we replace $\boldsymbol{1}_d$ with any point in the boundary of a unit cube, and this collectively shows that with probability tending to one, $\mathcal{C}_{\bm{\mu}}$ contains a hyper-cube centered at $\widetilde{\boldsymbol{\mu}}$ of size $(Rd)^{-1}\rho \max_{1\leq k\leq d}R_{n,k,\gamma}$.

The proof of (iv) is similar to that of (i) with $\max_{1\leq k\leq d}R_{n,k,\gamma}$ replaced by $R_{n,\bm{0},\gamma}$. The proof of assertion (v) can be completed by following the arguments used in the proof of assertion (ii) with \eqref{eq:maxinequality} applied to the pair $f$ and $\widetilde{f}$.
\end{proof}

To prove the results in Section \ref{sec:bayes}, we need to first lay out some preliminary details. Define $f_{0,\boldsymbol{z}}(\boldsymbol{x}-\widetilde{\boldsymbol{\mu}})=f_0(\boldsymbol{x})$ to be the shifted true function. Let $\boldsymbol{\theta}_0=(\theta_{0,\boldsymbol{i}}:\boldsymbol{i}\leq\boldsymbol{m}_{\boldsymbol{\alpha}})^T$ be a random vector such that $f_{\boldsymbol{\theta}_0}(\boldsymbol{x}-\widetilde{\boldsymbol{\mu}})=T_{\boldsymbol{\mu}_0}f_0(\boldsymbol{x})$, where $f_{\boldsymbol{\theta}}$ is from \eqref{eq:thetaf} and $T_{\boldsymbol{\mu}_0}f_0(\boldsymbol{x})$ is the Taylor polynomial of order $\boldsymbol{m}_{\boldsymbol{\alpha}}$ by expanding $f_0$ around $\boldsymbol{\mu}_0$, that is,
\begin{align}\label{eq:2theta0}
\sum_{\boldsymbol{i}\leq\boldsymbol{m}_{\boldsymbol{\alpha}}}\theta_{0,\boldsymbol{i}}(\boldsymbol{x}-\widetilde{\boldsymbol{\mu}})^{\boldsymbol{i}}
=f_0(\boldsymbol{\mu}_0)+\sum_{\boldsymbol{i}\leq\boldsymbol{m}_{\boldsymbol{\alpha}},|\boldsymbol{i}|\geq2}\frac{D^{\boldsymbol{i}}f_0(\boldsymbol{\mu}_0)}{\boldsymbol{i}!}(\boldsymbol{x}-\boldsymbol{\mu}_0)^{\boldsymbol{i}},
\end{align}
where $\nabla f_0(\boldsymbol{\mu}_0)=\boldsymbol{0}$ by Assumption 2. Hence $\boldsymbol{\theta}_0$ can be thought of as the true $\boldsymbol{\theta}$ by projecting $f_0$ onto the space of polynomials of order $\boldsymbol{m}_{\boldsymbol{\alpha}}$. Note that $\boldsymbol{\theta}_0$ is random and depends on $\widetilde{\boldsymbol{\mu}}$, $\boldsymbol{\mu}_0$ and $f_0$. By applying $D^{\boldsymbol{i}}$ on both sides of \eqref{eq:2theta0} and evaluating at $\boldsymbol{x}=\widetilde{\boldsymbol{\mu}}$, we have $\boldsymbol{i}!\theta_{0, \boldsymbol{i}}=D^{\boldsymbol{i}}T_{\boldsymbol{\mu}_0}f_0(\widetilde{\boldsymbol{\mu}})$. Note that since $D^{\boldsymbol{i}}f_0(\boldsymbol{x}),\boldsymbol{i}\leq\boldsymbol{m}_{\boldsymbol{\alpha}}$, are continuous by Assumption 2, they are bounded over $\{\boldsymbol{\mu}: |\mu_k-\widetilde{\mu}_k|\leq\delta_{n,k}, k=1,\dotsc,d\}$, and this implies that for any $\boldsymbol{i}\leq\boldsymbol{m}_{\boldsymbol{\alpha}}$, $|\theta_{0,\boldsymbol{i}}|=O_{P_0}(1)$ uniformly over $\|f_0\|_{\boldsymbol{\alpha},\infty}\leq R$. The design matrix $\boldsymbol{Z}$ is generated using i.i.d. uniform samples and by Lemma \ref{lem:random}, we know that $\boldsymbol{Z}^T\boldsymbol{Z}$ is invertible with probability going to $1$ as $n\rightarrow\infty$. In the actual computation, the invertibility of $\boldsymbol{Z}^T\boldsymbol{Z}$ is not an important issue as there is the presence of the prior covariance matrix $\boldsymbol{V}$ to serve as a regularization factor, and $\boldsymbol{Z}^T\boldsymbol{Z}+\boldsymbol{V}^{-1}$ in \eqref{eq:ptheta} is always invertible by our choice of $\boldsymbol{V}$.

As a consequence of Theorem \ref{th:frate}, $D_k\widetilde{f}$ converges uniformly to $D_kf_0$ at the rate $\epsilon_{n,k}$, and by \eqref{eq:mutilde}, this translates to $\|\widetilde{\boldsymbol{\mu}}-\boldsymbol{\mu}_0\|=O_{P_0}(\epsilon_n)$. Let $\boldsymbol{F}_0=(f_0(\boldsymbol{x}_1),\dotsc,f_0(\boldsymbol{x}_{n_2}))^T$ where $\{\boldsymbol{x}_1,\dotsc,\boldsymbol{x}_{n_2}\}$ is the original (unshifted) second stage samples. Note that $(\boldsymbol{Z\theta}_0)_i=f_{\boldsymbol{\theta}_0}(\boldsymbol{x}_i-\widetilde{\boldsymbol{\mu}})=T_{\boldsymbol{\mu}_0}f_0(\boldsymbol{x}_i)$ and under the assumption of \eqref{eq:isoholder}, we have
\begin{align}\label{eq:2approx}
\|\boldsymbol{F}_0-\boldsymbol{Z\theta}_0\|_\infty&=\max_{1\leq i\leq n_2}|f_0(\boldsymbol{x}_i)-T_{\boldsymbol{\mu}_0}f_0(\boldsymbol{x}_i)|\lesssim\max_{1\leq i\leq n_2}\sum_{k=1}^d|x_{ik}-\mu_{0k}|^{\alpha_k}\nonumber\\
&\lesssim\max_{1\leq i\leq n_2}\sum_{k=1}^d|x_{ik}-\widetilde{\mu}_k|^{\alpha_k}+\sum_{k=1}^d|\widetilde{\mu}_k-\mu_{0k}|^{\alpha_k}\lesssim\sum_{k=1}^d\delta_{n,k}^{\alpha_k}
\end{align}
uniformly over $\|f_0\|_{\boldsymbol{\alpha},\infty}\leq R$. The second line follows from the inequality $|x+y|^r\leq\max\{1,2^{r-1}\}(|x|^r+|y|^r)$, while the last inequality is due to $|x_{ik}-\widetilde{\mu}_k|\leq\delta_{n,k}$ almost surely since $x_{ik}-\widetilde{\mu}_k\sim\mathrm{Uniform}(-\delta_{n,k},\delta_{n,k})$; and $|\widetilde{\mu}_k-\mu_{0k}|\leq\|\widetilde{\boldsymbol{\mu}}-\boldsymbol{\mu}_0\|=O_{P_0}(\epsilon_n)=o_{P_0}(\delta_{n,k})$ as argued above, where $\epsilon_n=o(\delta_{n,k}),k=1,\dotsc,d$ was by our choice in \eqref{eq:delta}.

We break the proof of Theorem \ref{th:2M} into a series of steps. First, let us enumerate the elements of $\{\boldsymbol{i}:\boldsymbol{i}\leq\boldsymbol{m}_{\boldsymbol{\alpha}}\}$ as $\{\boldsymbol{i}_0,\dotsc,\boldsymbol{i}_W\}$ with $W+1=\prod_{k=1}^d\alpha_k$. For the rest of this section, we follow this indexing convention and index the entries of vectors $\boldsymbol{\xi}$, $\boldsymbol{\theta}$ and $\boldsymbol{\theta}_0$ by elements of $\{\boldsymbol{i}:\boldsymbol{i}\leq\boldsymbol{m}_{\boldsymbol{\alpha}}\}$. For matrices $\boldsymbol{Z}^T\boldsymbol{Z}$ and $\boldsymbol{V}$, we enumerate their rows and columns starting from $0$ and ending at $W$. We note that in our first and second stage sampling plans, $n_1\asymp n\asymp n_2$.

The first key step is to derive sharp upper bounds for the posterior mean and variance, which will involve upper bounding the entries of $(\boldsymbol{Z}^T\boldsymbol{Z})^{-1}$. These calculations are made simpler by centering the design points so that they are uniformly distributed around zero in each co-ordinate, and $(\boldsymbol{Z}^T\boldsymbol{Z})^{-1}$ will not depend on $\widetilde{\boldsymbol{\mu}}$. The following lemma describes the asymptotic behavior of the entries of $(\boldsymbol{Z}^T\boldsymbol{Z})^{-1}$ when the second stage samples are collected under uniform random sampling. In the following, we define $\boldsymbol{\delta}_n=(\delta_{n,1},\dotsc,\delta_{n,d})^T$ and write $\boldsymbol{\delta}_n^{\boldsymbol{i}}$ to mean $\prod_{k=1}^d\delta_{n,k}^{i_k}$.

\begin{lemma}\label{lem:random}
As $n\rightarrow\infty$, $\boldsymbol{Z}^T\boldsymbol{Z}$ is invertible with probability tending to $1$. Moreover for $a,b=0,\dotsc,W$, we have $[(\boldsymbol{Z}^T\boldsymbol{Z})^{-1}]_{ab}=O_{P}\left(n^{-1}\boldsymbol{\delta}_n^{-(\boldsymbol{i}_a+\boldsymbol{i}_b)}\right)$.
\end{lemma}

\begin{proof} 
After centering, second stage samples are distributed as $\boldsymbol{z}_i=(z_{i1},\dotsc,z_{id})^T\overset{\mathrm{i.i.d.}}{\sim}\mathrm{Uniform}\left(\prod_{k=1}^d[-\delta_{n,k},\delta_{n,k}]\right), i=1,\dotsc,n_2$, and $z_{ik}\sim\mathrm{Uniform}[-\delta_{n,k},\delta_{n,k}]$. Thus,
\begin{align}\label{eq:a}
\boldsymbol{Z}^T\boldsymbol{Z}=n_2\begin{pmatrix}
a_{00}&a_{01}\boldsymbol{\delta}_n^{\boldsymbol{i}_1}&\cdots& a_{0W}\boldsymbol{\delta}_n^{\boldsymbol{i}_W}\\
a_{10}\boldsymbol{\delta}_n^{\boldsymbol{i}_1}&a_{11}\boldsymbol{\delta}_n^{2\boldsymbol{i}_1}&\cdots& a_{1W}\boldsymbol{\delta}_n^{\boldsymbol{i}_1+\boldsymbol{i}_W}\\
\vdots&\vdots&\ddots&\vdots\\
a_{W0}\boldsymbol{\delta}_n^{\boldsymbol{i}_W}&a_{W1}\boldsymbol{\delta}_n^{\boldsymbol{i}_W+\boldsymbol{i}_1}&\cdots& a_{WW}\boldsymbol{\delta}_n^{2\boldsymbol{i}_W}
\end{pmatrix}:=n_2\boldsymbol{\Delta A\Delta}
\end{align}
with the $(i,j)$-entry of $\boldsymbol{A}$ being $a_{ij}=n_2^{-1}\sum_{k=1}^{n_2}\boldsymbol{U}_k^{\boldsymbol{i}_i}\boldsymbol{U}_k^{\boldsymbol{i}_j}$ where $\boldsymbol{U}_k=(U_{k1},\dotsc,U_{kd})^T\overset{\mathrm{i.i.d.}}{\sim}\mathrm{Uniform}[-1,1]^d$, and $\boldsymbol{\Delta}=\mathrm{diag}\left\{\boldsymbol{\delta}_n^{\boldsymbol{i}_j}:j=0,\dotsc,W\right\}$.
Define $\mathbb{U}=(\boldsymbol{U}^{\boldsymbol{i}_0},\dotsc,\boldsymbol{U}^{\boldsymbol{i}_W})^T$ for $\boldsymbol{U}=(U_1,\dotsc,U_d)^T\sim\mathrm{Uniform}[-1,1]^d$. By the law of large numbers, we have that $\boldsymbol{A}$ converges in probability to $\mathrm{E}\mathbb{U}\mathbb{U}^T$ entry-wise, and hence  $\mathrm{E}\mathbb{UU}^T-\epsilon\boldsymbol{I}\leq\boldsymbol{A}\leq\mathrm{E}\mathbb{UU}^T+\epsilon\boldsymbol{I}$ for a sufficiently small $\epsilon>0$. Observe that entries of $\mathrm{E}\mathbb{U}\mathbb{U}^T$ are mixed moments of $U\sim \mathrm{Uniform}[-1,1]$ and hence is positive definite. Then $\mathrm{E}\mathbb{UU}^T-\epsilon\boldsymbol{I}$ is invertible when $\epsilon$ is smaller than the minimum eigenvalue of $\mathrm{E}\mathbb{UU}^T$. Thus for sufficiently small $\epsilon>0$,
\begin{align*}
n_2^{-1}\Delta^{-1}(\mathrm{E}\mathbb{UU}^T+\epsilon\boldsymbol{I})^{-1}\Delta^{-1}\leq(\boldsymbol{Z}^T\boldsymbol{Z})^{-1}\leq n_2^{-1}\Delta^{-1}(\mathrm{E}\mathbb{UU}^T-\epsilon\boldsymbol{I})^{-1}\Delta^{-1}.
\end{align*}
Let $u^{ij}$ be the $(i,j)$th entry of $(\mathrm{E}\mathbb{U}\mathbb{U}^T)^{-1}$ and recall that $n_2\geq cn$ for some constant $c>0$. It then follows that for $a=0,\dotsc,W$, $[(\boldsymbol{Z}^T\boldsymbol{Z})^{-1}]_{aa}$ is $O_P(n_2^{-1}u^{aa}\boldsymbol{\delta}_n^{-(\boldsymbol{i}_a+\boldsymbol{i}_a)})=O_P(n^{-1}\boldsymbol{\delta}_n^{-(\boldsymbol{i}_a+\boldsymbol{i}_a)})$. Using the fact that for positive definite $\boldsymbol{G}$, $g_{ij}\leq\sqrt{g_{ii}g_{jj}}$ by the Cauchy-Schwarz inequality, $[(\boldsymbol{Z}^T\boldsymbol{Z})^{-1}]_{ab}$ with $a,b=0,\dotsc,W$ is bounded above by
\begin{align*}
\sqrt{[(\boldsymbol{Z}^T\boldsymbol{Z})^{-1}]_{aa}[(\boldsymbol{Z}^T\boldsymbol{Z})^{-1}]_{bb}}=O_P\left(n^{-1}\boldsymbol{\delta}_n^{-(\boldsymbol{i}_a+\boldsymbol{i}_b)}\right).\qquad\qedhere
\end{align*}
\end{proof}

\begin{proof}[Proof of Proposition 5.1]
By the triangle inequality, $|\widetilde{\sigma}^2_{*}-\sigma_0^2|\leq|\widetilde{\sigma}_1^2-\sigma_0^2|+|\widetilde{\sigma}_2^2-\sigma_0^2|$. By (a) of Proposition 9.5, the first term is $O_{P_0}(\max\{n^{-1/2},n^{-2\alpha^{*}/(2\alpha^{*}+d)}\})$. To bound the second term, let $\boldsymbol{U}=(\boldsymbol{ZVZ}^T+\boldsymbol{I}_{n})^{-1}$. By equation (33) of page 355 in \citet{expectquad},
\begin{align}
|\mathrm{E}(\widetilde{\sigma}_2^2|\boldsymbol{\theta}_0)-\sigma_0^2|&=|n^{-1}\sigma_0^2\mathrm{tr}(\boldsymbol{U})-\sigma_0^2|+n^{-1}(\boldsymbol{F}_0-\boldsymbol{Z\xi})^T\boldsymbol{U}(\boldsymbol{F}_0-\boldsymbol{Z\xi})\nonumber\\
&\lesssim n^{-1}[\mathrm{tr}(\boldsymbol{I}_{n}-\boldsymbol{U})+(\boldsymbol{F}_0-\boldsymbol{Z\theta}_0)^T
\boldsymbol{U}(\boldsymbol{F}_0-\boldsymbol{Z\theta}_0)\label{eq:2tsigma0bias}\\
&\quad+(\boldsymbol{Z\theta}_0-\boldsymbol{Z\xi})^T
\boldsymbol{U}(\boldsymbol{Z\theta}_0-\boldsymbol{Z\xi})],\nonumber
\end{align}
where we have used $(\boldsymbol{x}+\boldsymbol{y})^T\boldsymbol{G}(\boldsymbol{x}+\boldsymbol{y})\leq2\boldsymbol{x}^T\boldsymbol{Gx}+2\boldsymbol{y}^T\boldsymbol{Gy}$ for any matrix $\boldsymbol{G}\geq\boldsymbol{0}$ (Cauchy-Schwarz and the geometric-arithmetic inequalities). Let $\boldsymbol{P}_{\boldsymbol{Z}}=\boldsymbol{Z}(\boldsymbol{Z}^T\boldsymbol{Z})^{-1}\boldsymbol{Z}^T$ be the orthogonal projection matrix. For matrices $\boldsymbol{Q},\boldsymbol{C},\boldsymbol{T},\boldsymbol{W}$, the binomial inverse theorem (see Theorem 18.2.8 of \citet{davidmatrix}) says that
\begin{align*}
(\boldsymbol{Q}+\boldsymbol{CTW})^{-1}=\boldsymbol{Q}^{-1}-\boldsymbol{Q}^{-1}\boldsymbol{C}(\boldsymbol{T}^{-1}+\boldsymbol{WQ}^{-1}\boldsymbol{C})^{-1}\boldsymbol{WQ}^{-1}.
\end{align*}
Applying the above twice to $\boldsymbol{U}$ yields
\begin{align}\label{eq:2tbiasinter}
(\boldsymbol{ZVZ}^T+\boldsymbol{I}_{n})^{-1}&=\boldsymbol{I}_{n}-\boldsymbol{Z}
(\boldsymbol{Z}^T\boldsymbol{Z}+\boldsymbol{V}^{-1})^{-1}\boldsymbol{Z}^T=\boldsymbol{I}_{n}-\boldsymbol{P}_{\boldsymbol{Z}}
+\boldsymbol{M},
\end{align}
where $\boldsymbol{M}=\boldsymbol{Z}(\boldsymbol{Z}^T\boldsymbol{Z})^{-1}
[\boldsymbol{V}+(\boldsymbol{Z}^T\boldsymbol{Z})^{-1}]^{-1}(\boldsymbol{Z}^T\boldsymbol{Z})^{-1}\boldsymbol{Z}^T
\geq\boldsymbol{0}$. Hence the first term in \eqref{eq:2tsigma0bias} is $n^{-1}\mathrm{tr}(\boldsymbol{P}_{\boldsymbol{Z}}-\boldsymbol{M})\leq n^{-1}\mathrm{tr}(\boldsymbol{P}_{\boldsymbol{Z}})
=(W+1)/n$. Note that $\boldsymbol{U}\leq\boldsymbol{I}_{n}$ since $\boldsymbol{ZVZ}^T\geq\boldsymbol{0}$, and the second term in \eqref{eq:2tsigma0bias} is bounded by
\begin{align*}
n^{-1}\|\boldsymbol{U}\|_{(2,2)}\|\boldsymbol{F}_0-\boldsymbol{Z\theta}_0\|^2
\leq\|\boldsymbol{F}_0-\boldsymbol{Z\theta}_0\|_\infty^2
\lesssim\sum_{k=1}^d\delta_{n,k}^{2\alpha_k},
\end{align*}
in view of (8.3). By \eqref{eq:2tbiasinter} and $(\boldsymbol{I}-\boldsymbol{P}_{\boldsymbol{Z}})\boldsymbol{Z}=\boldsymbol{0}$, the last term in \eqref{eq:2tsigma0bias} is $n^{-1}(\boldsymbol{\theta}_0-\boldsymbol{\xi})^T[\boldsymbol{V}+(\boldsymbol{Z}^T\boldsymbol{Z})^{-1}]^{-1}
(\boldsymbol{\theta}_0-\boldsymbol{\xi})
\leq n^{-1}\sum_{j=0}^W\boldsymbol{\delta}_n^{\boldsymbol{i}_j}(\theta_{0,\boldsymbol{i}_j}-\xi_{\boldsymbol{i}_j})^2=O_{P_0}(n^{-1})$, since $\delta_{n,k}=o(1),k=1,\dotsc,d$, $\theta_{0,\boldsymbol{i}_j}=O_{P_0}(1)$ and $\xi_{\boldsymbol{i}_j}=O(1)$ by assumption on the prior for any $0\leq j\leq W$. Combining the three bounds established into \eqref{eq:2tsigma0bias}, we obtain $|\mathrm{E}(\widetilde{\sigma}_2^2|\boldsymbol{\theta}_0)-\sigma_0^2|\lesssim n^{-1}+\sum_{k=1}^d\delta_{n,k}^{2\alpha_k}$.

We write $n\widetilde{\sigma}_2^2=(\boldsymbol{F}_0-\boldsymbol{Z\xi})^T\boldsymbol{U}(\boldsymbol{F}_0-\boldsymbol{Z\xi})+2(\boldsymbol{F}_0-\boldsymbol{Z\xi})^T\boldsymbol{U\varepsilon}
+\boldsymbol{\varepsilon}^T\boldsymbol{U\varepsilon}$ by substituting $\boldsymbol{Y}=\boldsymbol{F}_0+\boldsymbol{\varepsilon}$. Observe that $\boldsymbol{\varepsilon}$ and $\boldsymbol{\theta}_0$ are independent by definition. Using the inequality $\mathrm{Var}(A_1+A_2)\leq2\mathrm{Var}(A_1)+2\mathrm{Var}(A_2)$ (from Cauchy-Schwarz and geometric-arithmetic inequalities), we conclude that $\mathrm{Var}(\widetilde{\sigma}_2^2|\boldsymbol{\theta}_0)$ is bounded up to a constant multiple by
\begin{align}\label{eq:2tsigma0var}
&n^{-2}[(\boldsymbol{F}_0-\boldsymbol{Z\theta}_0)^T\boldsymbol{U}^2(\boldsymbol{F}_0-\boldsymbol{Z\theta}_0)+(\boldsymbol{Z\theta}_0-\boldsymbol{Z\xi})^T\boldsymbol{U}^2(\boldsymbol{Z\theta}_0-\boldsymbol{Z\xi})
+\mathrm{Var}(\boldsymbol{\varepsilon}^T\boldsymbol{U}\boldsymbol{\varepsilon})].
\end{align}
In view of (8.3) and $\boldsymbol{U}\leq\boldsymbol{I}_{n}$, the first term is bounded by $n^{-2}\|\boldsymbol{U}\|_{(2,2)}^2\|\boldsymbol{F}_0-\boldsymbol{Z\theta}_0\|^2
\leq n^{-1}\|\boldsymbol{F}_0-\boldsymbol{Z\theta}_0\|_\infty^2\lesssim n^{-1}\sum_{k=1}^d\delta_{n,k}^{2\alpha_k}$. Observe that since $\boldsymbol{V}\geq\boldsymbol{0}$,
\begin{align}\label{eq:2tinterbound}
\boldsymbol{Z}^T\boldsymbol{M}^2\boldsymbol{Z}&=[\boldsymbol{V}+(\boldsymbol{Z}^T\boldsymbol{Z})^{-1}]^{-1}(\boldsymbol{Z}^T\boldsymbol{Z})^{-1}[\boldsymbol{V}+(\boldsymbol{Z}^T\boldsymbol{Z})^{-1}]^{-1}\nonumber\\
&\leq[\boldsymbol{V}+(\boldsymbol{Z}^T\boldsymbol{Z})^{-1}]^{-1}\leq\boldsymbol{Z}^T\boldsymbol{Z}.
\end{align}
Using \eqref{eq:2tbiasinter}, idempotency of $\boldsymbol{I}_{n}-\boldsymbol{P}_{\boldsymbol{Z}}$ and $(\boldsymbol{I}_{n}-\boldsymbol{P}_{\boldsymbol{Z}})\boldsymbol{Z}=\boldsymbol{0}$, the second term in \eqref{eq:2tsigma0var} is $n^{-2}(\boldsymbol{\theta}_0-\boldsymbol{\xi})^T\boldsymbol{Z}^T(\boldsymbol{I}_{n}-\boldsymbol{P}_{\boldsymbol{Z}}
+\boldsymbol{M})^2\boldsymbol{Z}(\boldsymbol{\theta}_0-\boldsymbol{\xi})$, which is
\begin{align}\label{eq:2varinter}
n^{-2}(\boldsymbol{\theta}_0-\boldsymbol{\xi})^T\boldsymbol{Z}^T\boldsymbol{M}^2\boldsymbol{Z}(\boldsymbol{\theta}_0-\boldsymbol{\xi})
\leq n^{-2}(\boldsymbol{\theta}_0-\boldsymbol{\xi})^T\boldsymbol{Z}^T\boldsymbol{Z}(\boldsymbol{\theta}_0-\boldsymbol{\xi}),
\end{align}
in view of \eqref{eq:2tinterbound}. By (8.4) in the proof of Lemma 8.1, we can write $\boldsymbol{Z}^T\boldsymbol{Z}=n_2\boldsymbol{\Delta A\Delta}$ where $\boldsymbol{\Delta}=\mathrm{diag}\{\boldsymbol{\delta}_n^{\boldsymbol{i}_j}:j=0,\dotsc,W\}$ and $\boldsymbol{A}\rightarrow\mathrm{E}\mathbb{U}\mathbb{U}^T$ in probability entry-wise, where $\mathbb{U}=(\boldsymbol{U}^{\boldsymbol{i}_0},\dotsc,\boldsymbol{U}^{\boldsymbol{i}_W})^T$ for $\boldsymbol{U}=(U_1,\dotsc,U_d)^T\sim\mathrm{Uniform}[-1,1]^d$. This gives  $\|\boldsymbol{A}\|_{(2,2)}\rightarrow\|\mathrm{E}\mathbb{U}\mathbb{U}^T\|_{(2,2)}$ in probability. The entries of $\mathrm{E}\mathbb{U}\mathbb{U}^T$ are mixed moments of $\mathrm{Uniform}[-1,1]$ and hence the matrix is nonsingular with $\|\mathrm{E}\mathbb{U}\mathbb{U}^T\|_{(2,2)}<\infty$. Since $\|\boldsymbol{\Delta}\|_{(2,2)}=1$ and $n_2\leq n$, the right hand side of \eqref{eq:2varinter} is bounded by
\begin{align*}
n_2n^{-2}\|\boldsymbol{A}\|_{(2,2)}\|\boldsymbol{\Delta}\|_{(2,2)}^2\|\boldsymbol{\theta}_0-\boldsymbol{\xi}\|^2=O_{P_0}(n^{-1}),
\end{align*}
because $\|\boldsymbol{\theta}_0-\boldsymbol{\xi}\|\le \|\boldsymbol{\theta}_0\|+\|\boldsymbol{\xi}\|=O_{P_0}(1)$. By Lemma A.10 of \citet{yoo2016} with $\|\boldsymbol{U}\|_{(2,2)}\leq1$ and Gaussian errors by Assumption 1, the last term in \eqref{eq:2tsigma0var} is $O(1/n)$. Combining this with the three bounds established above, we obtain $\mathrm{Var}(\widetilde{\sigma}_2^2|\boldsymbol{\theta}_0)=O_{P_0}(1/n)$. Therefore, the mean square error is $\mathrm{E}_0(\widetilde{\sigma}_2^2-\sigma_0^2)^2=\mathrm{E}\{\mathrm{E}[(\widetilde{\sigma}_2^2-\sigma_0^2)^2|\boldsymbol{\theta}_0]\}\lesssim n^{-1}+\sum_{k=1}^d\delta_{n,k}^{4\alpha_k}$.

To prove (b), observe that $\mathrm{E}(\sigma^2|\boldsymbol{Y})\lesssim n^{-1}+\widetilde{\sigma}_{*}^2$ and $\mathrm{Var}(\sigma^2|\boldsymbol{Y})\lesssim n^{-3}+n^{-1}\widetilde{\sigma}_{*}^4$. Therefore by Markov's inequality, the second stage posterior of $\sigma^2$ concentrates around the second stage empirical Bayes estimator $\widetilde{\sigma}_{*}^2$, and thus (b) will inherit the rate from (a) as established above.
\end{proof}

Let $\mathcal{K}_n:=[\sigma_0^2-m_n\xi_n,\sigma_0^2+m_n\xi_n]$ where $m_n$ is any sequence going to infinity (e.g., slowly varying such as $\log{n}$) and $\xi_n=\max\{n^{-1/2},n^{-2\alpha^{*}/(2\alpha^{*}+d)},\sum_{k=1}^d\delta_{n,k}^{2\alpha_k}\}$. Recall that $\mathcal{Q}=\{\boldsymbol{x}:|x_k|\leq\delta_{n,k}, k=1,\dotsc,d\}$ is the centered second stage credible set/sampling region.
\begin{lemma}\label{lem:2thetamse}
For any $\boldsymbol{i}\leq\boldsymbol{m}_{\boldsymbol{\alpha}}$ and $\sigma^2\in\mathcal{K}_n$,
\begin{align*}
\mathrm{E}[(\theta_{\boldsymbol{i}}-\theta_{0,\boldsymbol{i}})^2|\boldsymbol{Y},\sigma^2]=O_{P_0}\left[\prod_{k=1}^d\delta_{n,k}^{-2i_k}\left(\frac{1}{n}+\sum_{k=1}^d\delta_{n,k}^{2\alpha_k}\right)\right].
\end{align*}
\end{lemma}

\begin{proof}
Let $0\leq h\leq W$. Since $\boldsymbol{V}>\boldsymbol{0}$ by assumption, we have by Lemma \ref{lem:random} that $\sup_{\sigma^2\in\mathcal{K}_n}\mathrm{Var}(\theta_{\boldsymbol{i}_h}|\boldsymbol{Y},\sigma^2)$ is
\begin{align}\label{eq:2vartheta}
[\sigma_0^2+o(1)][(\boldsymbol{Z}^T\boldsymbol{Z}+\boldsymbol{V}^{-1})^{-1}]_{hh}
\lesssim[(\boldsymbol{Z}^T\boldsymbol{Z})^{-1}]_{hh}\lesssim n^{-1}\boldsymbol{\delta}_n^{-2\boldsymbol{i}_h}.
\end{align}
Now the bias for the conditional posterior mean in vector form is
\begin{align}\label{eq:2thetabias}
\mathrm{E}(\boldsymbol{\theta}|\boldsymbol{Y},\sigma^2)-\boldsymbol{\theta}_0&=(\boldsymbol{Z}^T\boldsymbol{Z}+\boldsymbol{V}^{-1})^{-1}(\boldsymbol{Z}^T\boldsymbol{Y}+\boldsymbol{V}^{-1}\boldsymbol{\xi})-\boldsymbol{\theta}_0\nonumber\\
&=(\boldsymbol{Z}^T\boldsymbol{Z}+\boldsymbol{V}^{-1})^{-1}[\boldsymbol{Z}^T\boldsymbol{\varepsilon}+\boldsymbol{Z}^T(\boldsymbol{F}_0-\boldsymbol{Z\theta}_0)+\boldsymbol{V}^{-1}(\boldsymbol{\xi}-\boldsymbol{\theta}_0)].
\end{align}
Following the same reasoning as in \eqref{eq:2vartheta}, the $h$th diagonal entry  of the covariance matrix $\sigma_0^2(\boldsymbol{Z}^T\boldsymbol{Z}+\boldsymbol{V}^{-1})^{-1}
\boldsymbol{Z}^T\boldsymbol{Z}(\boldsymbol{Z}^T\boldsymbol{Z}+\boldsymbol{V}^{-1})^{-1}$ of $(\boldsymbol{Z}^T\boldsymbol{Z}+\boldsymbol{V}^{-1})^{-1}
\boldsymbol{\varepsilon}$ is
\begin{align*}
\sigma_0^2[(\boldsymbol{Z}^T\boldsymbol{Z}+\boldsymbol{V}^{-1})^{-1}
\boldsymbol{Z}^T\boldsymbol{Z}(\boldsymbol{Z}^T\boldsymbol{Z}+\boldsymbol{V}^{-1})^{-1}]_{hh}
\lesssim n^{-1}\boldsymbol{\delta}_n^{-2\boldsymbol{i}_h}.
\end{align*}
Since $\mathrm{E}_0(\boldsymbol{\varepsilon})=\boldsymbol{0}$, it follows from Markov's inequality that the $h$th entry of $(\boldsymbol{Z}^T\boldsymbol{Z}+\boldsymbol{V}^{-1})^{-1}\boldsymbol{Z}^T\boldsymbol{\varepsilon}$ is $O_{P_0}\left(n^{-1/2}\boldsymbol{\delta}_n^{-\boldsymbol{i}_h}\right)$ for $0\leq h\leq W$.

Let $\beta_{ij}$ be the $(i,j)$th element of $(\boldsymbol{Z}^T\boldsymbol{Z}+\boldsymbol{V}^{-1})^{-1}$, $\kappa_{ij}$ be the $(i,j)$th element of $\boldsymbol{V}^{-1}$ and $\gamma_i$ be the $i$th entry of $\boldsymbol{F}_0-\boldsymbol{Z\theta}_0$. By \eqref{eq:2approx}, we have uniformly over $1\leq i\leq n$ that $|\gamma_i|\lesssim\sum_{k=1}^d\delta_{n,k}^{\alpha_k}$. Now using the fact that for positive definite $\boldsymbol{G}$, $g_{hj}\leq\sqrt{g_{hh}g_{jj}}$ by the Cauchy-Schwarz inequality, we have for $0\leq h,j\leq W$,
\begin{align*}
[(\boldsymbol{Z}^T\boldsymbol{Z}+\boldsymbol{V}^{-1})^{-1}]_{hj}&\leq\sqrt{
[(\boldsymbol{Z}^T\boldsymbol{Z}+\boldsymbol{V}^{-1})^{-1}]_{hh}[(\boldsymbol{Z}^T\boldsymbol{Z}+\boldsymbol{V}^{-1})^{-1}]_{jj}}\nonumber\\
&\leq\sqrt{[(\boldsymbol{Z}^T\boldsymbol{Z})^{-1}]_{hh}[(\boldsymbol{Z}^T\boldsymbol{Z})^{-1}]_{jj}}
\lesssim n^{-1}\boldsymbol{\delta}_n^{-(\boldsymbol{i}_h+\boldsymbol{i}_j)}.
\end{align*}
Since $\boldsymbol{z}_j\in \mathcal{Q}$, $j=1,\dotsc,n_2$, we have $|\boldsymbol{z}_j^{\boldsymbol{i}}|\leq\boldsymbol{\delta}_n^{\boldsymbol{i}}$ for $\boldsymbol{i}\leq\boldsymbol{m}_{\boldsymbol{\alpha}}$. Therefore, since $n_2\leq n$, $[(\boldsymbol{Z}^T\boldsymbol{Z}+\boldsymbol{V}^{-1})^{-1}\boldsymbol{Z}^T(\boldsymbol{F}_0-\boldsymbol{Z\theta}_0)]_h$ is
\begin{align*}
\beta_{h0}\sum_{j=1}^{n_2}\boldsymbol{z}_j^{\boldsymbol{i}_0}\gamma_j+\cdots+\beta_{hW}\sum_{j=1}^{n_2}\boldsymbol{z}_j^{\boldsymbol{i}_W}\gamma_j
\lesssim\boldsymbol{\delta}_n^{-\boldsymbol{i}_h}\sum_{k=1}^d\delta_{n,k}^{\alpha_k}.
\end{align*}
It remains to bound each entry of the last term in \eqref{eq:2thetabias}. Since $|\theta_{0,\boldsymbol{i}_j}-\xi_{0,\boldsymbol{i}_j}|\le |\theta_{0,\boldsymbol{i}_j}|+|\xi_{0,\boldsymbol{i}_j}|=O_{P_0}(1)$ for $j=0,\dotsc,W$, then Lemma \ref{lem:random} and the choice of $\boldsymbol{V}$ imply that  $[(\boldsymbol{Z}^T\boldsymbol{Z}+\boldsymbol{V}^{-1})^{-1}\boldsymbol{V}^{-1}(\boldsymbol{\xi}-\boldsymbol{\theta}_0)]_h$ is
\begin{align*}
\beta_{h0}\sum_{j=0}^W\kappa_{0j}(\xi_{\boldsymbol{i}_j}-\theta_{0,\boldsymbol{i}_j})+\cdots+\beta_{hW}\sum_{j=0}^W\kappa_{Wj}(\xi_{\boldsymbol{i}_j}-\theta_{0,\boldsymbol{i}_j})
\lesssim n^{-1}\boldsymbol{\delta}_n^{-\boldsymbol{i}_h}.
\end{align*}
Combining the bounds derived back into \eqref{eq:2thetabias}, the squared bias $\left[\mathrm{E}(\theta_{\boldsymbol{i}_h}|\boldsymbol{Y},\sigma^2)-\theta_{0,\boldsymbol{i}_h}\right]^2$ is $O_{P_0}\left[n^{-2}\boldsymbol{\delta}_n^{-2\boldsymbol{i}_h}+\boldsymbol{\delta}_n^{-2\boldsymbol{i}_h}\sum_{k=1}^d\delta_{n,k}^{2\alpha_k}\right]$. The result follows in view of the bounds established and \eqref{eq:2vartheta}.
\end{proof}

\begin{lemma}\label{th:2pointtheta}
Uniformly over $\|f_0\|_{\boldsymbol{\alpha},\infty}\leq R$, for any $\boldsymbol{r}\leq\boldsymbol{m}_{\boldsymbol{\alpha}}$, $\boldsymbol{x}\in \mathcal{Q}$ and $m_n\rightarrow\infty$,
$\mathrm{E}_0\Pi(|D^{\boldsymbol{r}}f_{\boldsymbol{\theta}}(\boldsymbol{x})- D^{\boldsymbol{r}}f_{0,\boldsymbol{z}}(\boldsymbol{x})|>m_n\epsilon_{n,\boldsymbol{r}}|\boldsymbol{Y})\rightarrow0$, where $\epsilon_{n,\boldsymbol{r}}:=\boldsymbol{\delta}_n^{-\boldsymbol{r}}(n^{-1/2}+\sum_{k=1}^d\delta_{n,k}^{\alpha_k})$.
\end{lemma}

\begin{proof}
In view of \eqref{eq:polyprime},
\begin{align*}
D^{\boldsymbol{r}}f_{\boldsymbol{\theta}}(\boldsymbol{x})-D^{\boldsymbol{r}}f_{\boldsymbol{\theta}_0}(\boldsymbol{x})
=\boldsymbol{r}!(\theta_{\boldsymbol{r}}-\theta_{0,\boldsymbol{r}})+\sum_{\boldsymbol{r}\leq\boldsymbol{i}\leq\boldsymbol{m}_{\boldsymbol{\alpha}},\boldsymbol{i}\neq\boldsymbol{r}}\frac{\boldsymbol{i}!}{(\boldsymbol{i}-\boldsymbol{r})!}(\theta_{\boldsymbol{i}}-\theta_{0,\boldsymbol{i}})\boldsymbol{x}^{\boldsymbol{i}-\boldsymbol{r}}.
\end{align*}
Observe that for any $\boldsymbol{x}\in \mathcal{Q}$, $|\boldsymbol{x}^{\boldsymbol{i}-\boldsymbol{r}}|\leq\boldsymbol{\delta}_n^{\boldsymbol{i}-\boldsymbol{r}}$. Also, by noting that $r_k\leq i_k\leq\alpha_k-1$ for $k=1,\dotsc,d$, we have both $\boldsymbol{r}!,\boldsymbol{i}!\leq\prod_{k=1}^d(\alpha_k-1)$. Using the fact $(\sum_{i=1}^n|b_i|)^p\leq n^{p-1}\sum_{i=1}^n|b_i|^p$ for $p\geq1$, $|D^{\boldsymbol{r}}f_{\boldsymbol{\theta}}(\boldsymbol{x})-D^{\boldsymbol{r}}f_{\boldsymbol{\theta}_0}(\boldsymbol{x})|^2$ is bounded above up to a constant multiple by
\begin{align}\label{eq:2primepoint}
|\theta_{\boldsymbol{r}}-\theta_{0,\boldsymbol{r}}|^2+\sum_{\boldsymbol{r}\leq\boldsymbol{i}\leq\boldsymbol{m}_{\boldsymbol{\alpha}},\boldsymbol{i}\neq\boldsymbol{r}}|\theta_{\boldsymbol{i}}-\theta_{0,\boldsymbol{i}}|^2\boldsymbol{\delta}_n^{2\boldsymbol{i}-2\boldsymbol{r}}.
\end{align}
Therefore, for any $\boldsymbol{r}\leq\boldsymbol{m}_{\boldsymbol{\alpha}}$ and any $\boldsymbol{x}\in \mathcal{Q}$, we have uniformly over $\|f_0\|_{\boldsymbol{\alpha},\infty}\leq R$ that $\mathrm{E}_0\sup_{\sigma^2\in\mathcal{K}_n}\mathrm{E}(|D^{\boldsymbol{r}}f_{\boldsymbol{\theta}}(\boldsymbol{x})-D^{\boldsymbol{r}}
f_{\boldsymbol{\theta}_0}(\boldsymbol{x})|^2|\boldsymbol{Y},\sigma^2)$ is bounded up to a constant multiple by
\begin{align}\label{eq:2varbound}
&\mathrm{E}_0\sup_{\sigma^2\in\mathcal{K}_n}\mathrm{E}[(\theta_{\boldsymbol{r}}-\theta_{0,\boldsymbol{r}})^2
|\boldsymbol{Y},\sigma^2]+\sum_{\boldsymbol{r}\leq\boldsymbol{i}\leq\boldsymbol{m}_{\boldsymbol{\alpha}},\boldsymbol{i}\neq\boldsymbol{r}}\boldsymbol{\delta}_n^{2\boldsymbol{i}-2\boldsymbol{r}}
\mathrm{E}_0\sup_{\sigma^2\in\mathcal{K}_n}\mathrm{E}[(\theta_{\boldsymbol{i}}-\theta_{0,\boldsymbol{i}})^2|\boldsymbol{Y},\sigma^2].
\end{align}
In view of Lemma \ref{lem:2thetamse}, the first term is bounded by $\boldsymbol{\delta}_n^{-2\boldsymbol{r}}(n^{-1}+\sum_{k=1}^d\delta_{n,k}^{2\alpha_k})$. By noting that the sum over $\{\boldsymbol{i}:\boldsymbol{r}\leq\boldsymbol{i}\leq\boldsymbol{m}_{\boldsymbol{\alpha}},\boldsymbol{i}\neq\boldsymbol{r}\}$ has at most $\prod_{k=1}^d\alpha_k$ terms, another application of Lemma \ref{lem:2thetamse} implies that the second term in \eqref{eq:2varbound} is bounded above by
\begin{align*}
\sum_{\boldsymbol{r}\leq\boldsymbol{i}\leq\boldsymbol{m}_{\boldsymbol{\alpha}},\boldsymbol{i}\neq\boldsymbol{r}}\boldsymbol{\delta}_n^{2\boldsymbol{i}-2\boldsymbol{r}}\left[\boldsymbol{\delta}_n^{-2\boldsymbol{i}}\left(\frac{1}{n}+\sum_{k=1}^d\delta_{n,k}^{2\alpha_k}\right)\right]
\lesssim\boldsymbol{\delta}_n^{-2\boldsymbol{r}}\left(\frac{1}{n}+\sum_{k=1}^d\delta_{n,k}^{2\alpha_k}\right).
\end{align*}
Using the inequality $|a+b|^r\leq\mathrm{max}(1,2^{r-1})(|a|^r+|b|^r),r>0$ and \eqref{eq:isoholder}, we have $|D^{\boldsymbol{r}}f_{\boldsymbol{\theta}_0}(\boldsymbol{x})-D^{\boldsymbol{r}}f_{0,\boldsymbol{z}}(\boldsymbol{x})|$ is
\begin{align}\label{eq:2pointbiasbound}
|D^{\boldsymbol{r}}T_{\boldsymbol{\mu}_0}(\boldsymbol{x}+\widetilde{\boldsymbol{\mu}})-D^{\boldsymbol{r}}f_0(\boldsymbol{x}+\widetilde{\boldsymbol{\mu}})|&\lesssim\sum_{k=1}^d|x_k+\widetilde{\mu}_k-\mu_{0,k}|^{\alpha_k-r_k}\nonumber\\
&\lesssim\sum_{k=1}^d\delta_{n,k}^{\alpha_k-r_k}+\sum_{k=1}^d|\widetilde{\mu}_k-\mu_{0,k}|^{\alpha_k-r_k},
\end{align}
and $\mathrm{E}_0|D^{\boldsymbol{r}}f_{\boldsymbol{\theta}_0}(\boldsymbol{x})-D^{\boldsymbol{r}}f_{0,\boldsymbol{z}}(\boldsymbol{x})|^2
\lesssim\sum_{k=1}^d\delta_{n,k}^{2\alpha_k-2r_k}$ uniformly in $\|f_0\|_{\boldsymbol{\alpha},\infty}\leq R$ by \eqref{eq:2approx}.

Define $P_{n,\boldsymbol{r}}(\boldsymbol{x}):=\mathrm{E}_0\sup_{\sigma^2\in\mathcal{K}_n}\mathrm{E}[(D^{\boldsymbol{r}}f_{\boldsymbol{\theta}}(\boldsymbol{x})-D^{\boldsymbol{r}}f_{0,\boldsymbol{z}}
(\boldsymbol{x}))^2|\boldsymbol{Y},\sigma^2]$. Combining all the bounds established and \eqref{eq:2varbound}, we have uniformly over $\|f_0\|_{\boldsymbol{\alpha},\infty}\leq R$,
\begin{align*}
P_{n,\boldsymbol{r}}(\boldsymbol{x}) &\lesssim \mathrm{E}_0\sup_{\sigma^2\in\mathcal{K}_n}\mathrm{E}(|D^{\boldsymbol{r}}f_{\boldsymbol{\theta}}(\boldsymbol{x})-D^{\boldsymbol{r}}
f_{\boldsymbol{\theta}_0}(\boldsymbol{x})|^2|\boldsymbol{Y},\sigma^2)+\mathrm{E}_0|D^{\boldsymbol{r}}f_{\boldsymbol{\theta}_0}(\boldsymbol{x})-D^{\boldsymbol{r}}f_{0,\boldsymbol{z}}(\boldsymbol{x})|^2\\
&\lesssim\boldsymbol{\delta}_n^{-2\boldsymbol{r}}\left(\frac{1}{n}+\sum_{k=1}^d\delta_{n,k}^{2\alpha_k}\right)+\sum_{k=1}^d\delta_{n,k}^{2\alpha_k-2r_k}\lesssim\epsilon_{n,\boldsymbol{r}}^2.
\end{align*}
By Proposition~\ref{prop:var2}, $P_0\left(\widetilde{\sigma}_{*}^2\in\mathcal{K}_n\right)\rightarrow1$ as $n\rightarrow\infty$ uniformly in $\|f_0\|_{\boldsymbol{\alpha},\infty}\leq R$. For the empirical Bayes posterior $\Pi(\cdot|\boldsymbol{Y})\equiv\Pi_{\widetilde{\sigma}_{*}}(\cdot|\boldsymbol{Y})$ and by Markov's inequality, we have for any $m_n\rightarrow\infty$,
\begin{align}\label{eq:2empbayes}
&\mathrm{E}_0\Pi_{\widetilde{\sigma}_{*}}(|D^{\boldsymbol{r}}f_{\boldsymbol{\theta}}(\boldsymbol{x})-D^{\boldsymbol{r}}f_{0,\boldsymbol{z}}(\boldsymbol{x})|>m_n\epsilon_{n,\boldsymbol{r}}|\boldsymbol{Y})
\leq\frac{P_{n,\boldsymbol{r}}(\boldsymbol{x})}{m_n^2\epsilon_{n,\boldsymbol{r}}^2}+o(1)\rightarrow0,
\end{align}
uniformly over $\|f_0\|_{\boldsymbol{\alpha},\infty}\leq R$. For the hierarchical Bayes procedure, we have for any $m_n\rightarrow\infty$ that $\mathrm{E}_0\Pi(|D^{\boldsymbol{r}}f_{\boldsymbol{\theta}}(\boldsymbol{x})-D^{\boldsymbol{r}}f_{0,\boldsymbol{z}}(\boldsymbol{x})|>m_n\epsilon_{n,\boldsymbol{r}}|\boldsymbol{Y})$ is uniformly over $\|f_0\|_{\boldsymbol{\alpha},\infty}\leq R$ bounded above by
\begin{equation}
\label{eq:2fullbayes}
P_{n,\boldsymbol{r}}(\boldsymbol{x})/(m_n\epsilon_{n,\boldsymbol{r}})^2+\mathrm{E}_0\Pi\left(\sigma^2 \not \in \mathcal{K}_n \middle|\boldsymbol{Y}\right).
\end{equation}
The first term is $o(1)$ since $P_{n,\boldsymbol{r}}(\boldsymbol{x})\lesssim\epsilon_{n,\boldsymbol{r}}^2$, while the second term goes to zero by Proposition~\ref{prop:var2}.
\end{proof}

An immediate consequence of the previous lemma is the following:

\begin{corollary}\label{cor:2thetainfty}
Uniformly in $\|f_0\|_{\boldsymbol{\alpha},\infty}\leq R$, we have for any $\boldsymbol{r}\leq\boldsymbol{m}_{\boldsymbol{\alpha}}$ and $m_n\rightarrow\infty$ that
$\mathrm{E}_0\Pi(\|D^{\boldsymbol{r}}f_{\boldsymbol{\theta}}- D^{\boldsymbol{r}}f_{0,\boldsymbol{z}}\|_\infty>m_n\epsilon_{n,\boldsymbol{r}}|\boldsymbol{Y})\rightarrow0$.
\end{corollary}

\begin{proof}
By (8.7), we have
\begin{align}
\|D^{\boldsymbol{r}}f_{\boldsymbol{\theta}}-D^{\boldsymbol{r}}f_{\boldsymbol{\theta}_0}\|_\infty&=\sup_{\boldsymbol{x}\in \mathcal{Q}}|D^{\boldsymbol{r}}f_{\boldsymbol{\theta}}(\boldsymbol{x})-D^{\boldsymbol{r}}f_{\boldsymbol{\theta}_0}(\boldsymbol{x})|\nonumber\\
&\lesssim|\theta_{\boldsymbol{r}}-\theta_{0,\boldsymbol{r}}|+\sum_{\boldsymbol{r}\leq\boldsymbol{i}\leq\boldsymbol{m}_{\boldsymbol{\alpha}},\boldsymbol{i}\neq\boldsymbol{r}}|\theta_{\boldsymbol{i}}-\theta_{0,\boldsymbol{i}}|\boldsymbol{\delta}_n^{\boldsymbol{i}-\boldsymbol{r}}.
\end{align}
Hence, the upper bound (8.8) is applicable and uniformly over $\|f_0\|_{\boldsymbol{\alpha},\infty}\leq R$, we will have $\mathrm{E}_0\sup_{\sigma^2\in\mathcal{K}_n}\mathrm{E}[\|D^{\boldsymbol{r}}f_{\boldsymbol{\theta}}-D^{\boldsymbol{r}}f_{\boldsymbol{\theta}_0}\|_\infty^2|\boldsymbol{Y},\sigma^2]\lesssim\delta_n^{-2\boldsymbol{r}}(n^{-1}+\sum_{k=1}^d\delta_{n,k}^{2\alpha_k})$. Moreover, since the bound in (8.9) is uniform for all $\boldsymbol{x}\in \mathcal{Q}$, this implies that $\mathrm{E}_0\|D^{\boldsymbol{r}}f_{\boldsymbol{\theta}_0}-D^{\boldsymbol{r}}f_{0,\boldsymbol{z}}\|_\infty^2\lesssim\sum_{k=1}^d\delta_{n,k}^{2\alpha_k-2r_k}$. Therefore, we conclude that uniformly over $\|f_0\|_{\boldsymbol{\alpha},\infty}\leq R$,
\begin{eqnarray*}
\lefteqn{\mathrm{E}_0\sup_{\sigma^2\in\mathcal{K}_n}\mathrm{E}[\|D^{\boldsymbol{r}}f_{\boldsymbol{\theta}}-D^{\boldsymbol{r}}f_{0,\boldsymbol{z}}\|_\infty^2|\boldsymbol{Y},\sigma^2]}\\
&&\lesssim  \mathrm{E}_0\sup_{\sigma^2\in\mathcal{K}_n}\mathrm{E}[\|D^{\boldsymbol{r}}f_{\boldsymbol{\theta}}-D^{\boldsymbol{r}}f_{\boldsymbol{\theta}_0}\|_\infty^2|\boldsymbol{Y},\sigma^2] +
\mathrm{E}_0\|D^{\boldsymbol{r}}f_{\boldsymbol{\theta}_0}-D^{\boldsymbol{r}}f_{0,\boldsymbol{z}}\|_\infty^2\\
&&\lesssim\boldsymbol{\delta}_n^{-2\boldsymbol{r}}\left(\frac{1}{n}+\sum_{k=1}^d\delta_{n,k}^{2\alpha_k}\right).
\end{eqnarray*}
The empirical and hierarchical posterior contraction rates then follow from (8.10) and (8.11) with absolute values replaced by sup-norms.
\end{proof}

We are now ready to prove Theorem \ref{th:2M}.

\begin{proof}[Proof of Theorem \ref{th:2M}]
We shall prove only the empirical Bayes case as the hierarchical Bayes case follows the same steps. Recall that $\boldsymbol{\mu}=\widetilde{\boldsymbol{\mu}}+\boldsymbol{\mu}_{\boldsymbol{z}}$. As a consequence of Theorem \ref{th:murate} and our choice of $\delta_{n,k},k=1,\dotsc,d$, we have $P_0(\boldsymbol{\mu}_0-\widetilde{\boldsymbol{\mu}}\in\mathcal{Q})\rightarrow1$. Therefore by \eqref{eq:muinequality},
\begin{align*}
\|\boldsymbol{\mu}-\boldsymbol{\mu}_0\|=\|\boldsymbol{\mu}_{\boldsymbol{z}}-(\boldsymbol{\mu}_0-\widetilde{\boldsymbol{\mu}})\|
\leq\frac{\sqrt{d}}{\lambda_0}\max_{1\leq k\leq d}\sup_{\boldsymbol{x}\in \mathcal{Q}}|D_kf_{\boldsymbol{\theta}}(\boldsymbol{x})-D_kf_{0,\boldsymbol{z}}(\boldsymbol{x})|.
\end{align*}
Let $\tau_{n,k}:=\delta_{n,k}^{-1}(n^{-1/2}+\sum_{k=1}^d\delta_{n,k}^{\alpha_k})$. Using this bound and Corollary \ref{cor:2thetainfty} with $\boldsymbol{r}=\boldsymbol{e}_k$, we have for any $m_n\rightarrow\infty$ that $\mathrm{E}_0\Pi(\|\boldsymbol{\mu}-\boldsymbol{\mu}_0\|>m_n\max_{1\leq k\leq d}\tau_{n,k}|\boldsymbol{Y})$ is bounded above by
\begin{align*}
\sum_{k=1}^d\mathrm{E}_0\Pi\left(\|D_k f_{\boldsymbol{\theta}}-D_kf_{0,\boldsymbol{z}}\|_\infty>\frac{\lambda_0}{\sqrt{d}}m_n\max_{1\leq k\leq d}\tau_{n,k}\middle|\boldsymbol{Y}\right)\rightarrow0.
\end{align*}

By definition, $M=f_{\boldsymbol{\theta}}(\boldsymbol{\mu}_{\boldsymbol{z}})$ and $M_0=f_{0,\boldsymbol{z}}(\boldsymbol{\mu}_0-\widetilde{\boldsymbol{\mu}})$. Then by \eqref{eq:maxinequality}, $|M-M_0|\leq\sup_{\boldsymbol{x}\in\mathcal{Q}}|f_{\boldsymbol{\theta}}(\boldsymbol{x})-f_{0,\boldsymbol{z}}(\boldsymbol{x})|$ since $P_0(\boldsymbol{\mu}_0-\widetilde{\boldsymbol{\mu}}\in\mathcal{Q})\rightarrow1$ as before. Therefore by Corollary \ref{cor:2thetainfty} with $\boldsymbol{r}=\boldsymbol{0}$, we have for $m_n\rightarrow\infty$, $\mathrm{E}_0\Pi[|M-M_0|>m_n(n^{-1/2}+\sum_{k=1}^d\delta_{n,k}^{\alpha_k})|\boldsymbol{Y}]\leq\mathrm{E}_0\Pi[\|f_{\boldsymbol{\theta}}-f_{0,\boldsymbol{z}}\|_\infty>m_n(n^{-1/2}+\sum_{k=1}^d\delta_{n,k}^{\alpha_k})|\boldsymbol{Y}]\rightarrow0$,
uniformly over $\|f_0\|_{\boldsymbol{\alpha},\infty}\leq R$.

To prove the last part, note that $\delta_{n,k}=n^{-1/(2\alpha_k)}, k=1,\dotsc,d$, comes from equating the two terms in the second stage rates for $\boldsymbol{\mu}$ and $M$, i.e., $n^{-1/2}=\sum_{k=1}^d\delta_{n,k}^{\alpha_k}$. To solve for $\delta_{n,k}$, take $\delta_{n,k}=(d^{-1}\delta_n)^{1/\alpha_k}$, where $\delta_n$ is a positive sequence in $n$ that does not depend on $k$. It follows that $\delta_n=n^{-1/2}$ and hence $\delta_{n,k}=n^{-1/(2\alpha_k)}$. The condition $\min_{1\leq k\leq d}\delta_{n,k}=\rho_n\epsilon_n$ or equivalently $\epsilon_n=o(\min_{1\leq k\leq d}\delta_{n,k})$ is fulfilled when $1/(2\underline{\alpha})<\alpha^{*}(1-\underline{\alpha}^{-1})/(2\alpha^{*}+d)$ for $\underline{\alpha}=\min_{1\leq k\leq d}\alpha_k$. By rearranging, we need $2\underline{\alpha}\alpha^{*}-4\alpha^{*}-d>0$. Since $\underline{\alpha}>2$ by Assumption 2, we have $2\underline{\alpha}\alpha^{*}-4\alpha^{*}-d=2(\underline{\alpha}-2)\alpha^{*}-d\geq2\underline{\alpha}^2-4\underline{\alpha}-d$. Thus, it suffices to find $\underline{\alpha}$ such that $2\underline{\alpha}^2-4\underline{\alpha}-d>0$ and this is satisfied if $\underline{\alpha}>1+\sqrt{1+d/2}$ under the constraint that $\underline{\alpha}>2$.
\end{proof}

\section{Appendix}\label{sec:appendix}
In this section, we collect some auxiliary results and technical lemmas that were used in several places to prove the main theorems in the previous section.

The next two results concern contraction rates and credible band coverage for the regression function $f$ and its derivatives, and they correspond to Theorems 4.4 and 5.3 of \citet{yoo2016} respectively.
\begin{theorem}\label{th:frate}
Let $J_{n,k}\asymp(n/\log{n})^{\alpha^{*}/\{\alpha_k(2\alpha^{*}+d)\}},k=1,\dotsc,d$. Then under Assumption 1, we have for any sequence $m_n\rightarrow\infty$,
\begin{align*}
\sup_{\|f_0\|_{\boldsymbol{\alpha},\infty}\leq R}\mathrm{E}_0\Pi(\|D^{\boldsymbol{r}}f-D^{\boldsymbol{r}}f_0\|_\infty>m_n(\log{n}/n)^{\alpha^{*}\{1-\sum_{k=1}^d(r_k/\alpha_k)\}/(2\alpha^{*}+d)}|\boldsymbol{Y})\rightarrow0.
\end{align*}
In particular, contraction rate for $f$ can be recovered by setting $\boldsymbol{r}=\boldsymbol{0}$.
\end{theorem}

Consider the simultaneous credible band $\{f:\|D^{\boldsymbol{r}}f-D^{\boldsymbol{r}}\widetilde{f}\|_\infty\leq\rho R_{n,\boldsymbol{r},\gamma}\}$, where the quantile $R_{n,\boldsymbol{r},\gamma}$ is chosen such that $\Pi(\|D^{\boldsymbol{r}}f-D^{\boldsymbol{r}}\widetilde{f}\|_\infty\leq R_{n,\boldsymbol{r},\gamma}|\boldsymbol{Y})=1-\gamma$ and $\rho>0$ is some large enough constant.
\begin{theorem}\label{th:fcred}
Let $J_{n,k}\asymp(n/\log{n})^{\alpha^{*}/\{\alpha_k(2\alpha^{*}+d)\}},k=1,\dotsc,d$. Then under Assumption 1,
\begin{enumerate}
\item $\inf_{\|f_0\|_{\boldsymbol{\alpha},\infty}\leq R}P_0(\|D^{\boldsymbol{r}}\widetilde{f}-D^{\boldsymbol{r}}f_0\|_\infty\leq\rho R_{n,\boldsymbol{r},\gamma})\rightarrow1$,
\item $R_{n,\boldsymbol{r},\gamma}\asymp(\log{n}/n)^{\alpha^{*}\{1-\sum_{k=1}^d(r_k/\alpha_k)\}/(2\alpha^{*}+d)}$ in $P_0$-probability.
\end{enumerate}
\end{theorem}
Credible bands for $f$ is recovered by $\boldsymbol{r}=\boldsymbol{0}$ and for $D^{\boldsymbol{e}_k}f\equiv D_kf$ by $\boldsymbol{r}=\boldsymbol{e}_k$, in the latter case we also write the radius as $R_{n,k,\gamma}=R_{n,\boldsymbol{e}_k,\gamma}$.

The result below was taken from (3.10) and (3.11) of \citet{yoo2016}, and it shows that B-splines despite being a non-orthonormal basis, are approximately orthogonal under our assumption on the design points.
\begin{lemma}\label{lem:BB}
Let $J=\prod_{k=1}^dJ_k$. If the design points were chosen such that \eqref{assump:cdf} holds, then for some constants $C_1,C_2,c_1,c_2>0$,
\begin{align}
C_1(n/J)&\leq\boldsymbol{B}^T\boldsymbol{B}\leq C_2(n/J),\label{eq:BB}\\
C_1(n/J)+c_2^{-1}\leq\lambda_{\mathrm{min}}(\boldsymbol{B}^T\boldsymbol{B}+\boldsymbol{\Omega}^{-1})&\leq\lambda_{\mathrm{max}}(\boldsymbol{B}^T\boldsymbol{B}+\boldsymbol{\Omega}^{-1})\leq C_2(n/J)+c_1^{-1}.\label{eq:BBO}
\end{align}
\end{lemma}

\begin{lemma}\label{lem:wr}
Let $\boldsymbol{W}_{\boldsymbol{r}}$ be the finite difference matrix as seen in \eqref{eq:fprior}. Under the quasi-uniformity of the knot distribution, we have $\lambda_{\mathrm{max}}(\boldsymbol{W}_{\boldsymbol{r}}\boldsymbol{W}_{\boldsymbol{r}}^T)=O(\prod_{k=1}^dJ_k^{2r_k})$.
\end{lemma}

\begin{proof}
Note that $\lambda_{\mathrm{max}}(\boldsymbol{W}_{\boldsymbol{r}}\boldsymbol{W}_{\boldsymbol{r}}^T)=\lambda_{\mathrm{max}}(\boldsymbol{W}_{\boldsymbol{r}}^T\boldsymbol{W}_{\boldsymbol{r}})
\leq\|\boldsymbol{W}_{\boldsymbol{r}}^T\boldsymbol{W}_{\boldsymbol{r}}\|_{(2,2)}\lesssim\prod_{k=1}^dJ_k^{2r_k}$, where the last upper bound was computed in (7.15) of \citet{yoo2016}.
\end{proof}

The following proposition shows that the single stage or the first stage (in the setting of two-stage procedure) empirical or hierarchical Bayes estimator of $\sigma^2$ is consistent. Note that this result corresponds to Proposition 4.1 of \citet{yoo2016}.
\begin{proposition}[First stage error variance]\label{prop:var1}
Suppose $J_{n,k}\asymp(n/\log{n})^{\alpha^{*}/\{\alpha_k(2\alpha^{*}+d)\}}$. Then uniformly over $\|f_0\|_{\boldsymbol{\alpha},\infty}\leq R$,
\begin{itemize}
\item [(a)] First stage empirical Bayes estimator $\widetilde{\sigma}_1^2$ converges to $\sigma_0^2$ under $P_0$-probability at the rate $\max\{n^{-1/2},n^{-2\alpha^{*}/(2\alpha^{*}+d)}\}$ .
\item [(b)] If inverse gamma prior is used, first stage posterior for $\sigma^2$ contracts to $\sigma_0^2$ at the same rate.
\item [(c)] If the prior used has continuous and positive density on $(0,\infty)$, then the first stage posterior for $\sigma$ is consistent.
\end{itemize}
\end{proposition}

\bibliographystyle{apa}
\bibliography{referencemax}
\end{document}